\documentclass[12pt]{amsart}

\usepackage{amssymb, amscd}
\usepackage{latexsym,epsfig}
\usepackage[all]{xy}
\usepackage[pdf]{pstricks}
\usepackage{hyperref}

\numberwithin{equation}{section}
\def\today{\ifcase\month\or Jan\or Febr\or  Mar\or  Apr\or May\or Jun\or  Jul\or Aug\or  Sep\or  Oct\or Nov\or  Dec\or\fi \space\number\day, \number\year}

\def\C{{\mathcal C}}
\def\H{{\mathcal H}}
\def\M{{\mathcal M}}

\newcommand{\CC}{\mathbb C}
\newcommand{\EE}{\mathbb E}

\newcommand{\PP}{\mathbb P}
\newcommand{\QQ}{\mathbb Q}

\newcommand{\ZZ}{\mathbb Z}

\newcommand\langepijl[1]{\buildrel {#1} \over \longrightarrow}

\def\tH{\widetilde{\mathcal H}}
\def\tC{\widetilde{\mathcal C}}
\def\tP{\widetilde{\PP}}
\def\bM{\overline{\mathcal M}}
\def\bM{\overline{\M}}

\numberwithin{equation}{section}
\newtheorem{theorem}{Theorem}[section]
\newtheorem{lemma}[theorem]{Lemma}
\newtheorem{proposition}[theorem]{Proposition}
\newtheorem{corollary}[theorem]{Corollary}

\newtheorem{definition-lemma}[theorem]{Definition-Lemma}
\theoremstyle{definition}

\theoremstyle{remark}
\newtheorem{remark}[theorem]{Remark}

\begin{document}

\title[Effective divisors and modular forms]{Effective divisors on projectivized Hodge \\
bundles and Modular Forms}
\begin{abstract}
We construct vector-valued modular forms on moduli spaces of curves and abelian varieties
using effective divisors in projectivized Hodge bundles over moduli of curves.
Cycle relations tell us the weight of these modular forms.
In particular we construct basic modular forms for genus $2$ and $3$. 
We also discuss modular forms on the moduli of hyperelliptic curves.
In that case the relative canonical bundle is a pull back of a line bundle on
a ${\PP}^1$-bundle over the moduli of hyperelliptic curves
and we extend that line bundle to a compactification
so that its push down is (close to) the Hodge bundle and
use this to construct modular forms.
In an appendix we use our method to calculate divisor classes
in the dual projectivized $k$-Hodge bundle determined by Gheorghita-Tarasca
and by Korotkin-Sauvaget-Zograf.

\end{abstract}

\author{Gerard van der Geer}
\author{Alexis Kouvidakis}
\maketitle

\begin{section}{Introduction}\label{sec-intro}
Moduli spaces of curves and of abelian varieties come with a natural
vector bundle, the Hodge bundle ${\EE}$. Starting from this vector bundle
one can
construct other natural vector bundles by applying Schur functors,
like ${\rm Sym}^n({\EE})$ or $\det({\EE})^{\otimes m}$. Sections of
such bundles are called modular forms. For example, for the moduli space
$\mathcal{A}_g$ of principally polarized abelian varieties of dimension $g$
these are Siegel modular forms,  and for the moduli space $\mathcal{M}_g$
of curves of genus $g$ these are Teichm\"uller modular forms. If the Schur functor
corresponds to an irreducible representation $\rho$ we say that a section
of ${\EE}_{\rho}$ is a modular form of weight $\rho$. The Hodge bundle extends
to appropriate compactifications of such moduli spaces and in many cases
the sections also extend automatically to the compactifications,
e.g. for $\mathcal{A}_g$ with  $g\geq 2$ by the so-called Koecher principle.

In this paper we try to construct modular forms in a geometric way. 
It is well-known that an effective divisor on $\mathcal{A}_g$ 
or on $\overline{\mathcal{M}}_g$ with $g\geq 2$
representing the cycle class $m\lambda$ with $\lambda=c_1(\det({\EE}))$ and $m\in {\ZZ}_{>0}$
yields a scalar-valued modular form of weight $m$, that is, a section of $\det({\EE})^{\otimes m}$. We will exploit explicit
effective divisors on projectivized vector bundles to construct vector-valued 
modular forms. In particular, we will  construct in this way certain
modular forms that play a pivotal role in low genera.

For example, in the case of $g=2$ there is the modular form $\chi_{6,8}$,
a section of ${\rm Sym}^6({\EE}) \otimes \det({\EE})^8$, that appeared in \cite{CFG1} as follows.
Recall that the Torelli morphism
$
\mathcal{M}_2 \hookrightarrow \mathcal{A}_2
$
has dense image and we have an equality of standard compactifications
$\overline{\mathcal{M}}_2= \tilde{\mathcal{A}}_2$. The
moduli space $\mathcal{M}_2$ has another description as a stack quotient.
This derives from the fact that a smooth complete curve of genus $2$
over a field $k$ of characteristic not $2$ is a double cover of ${\PP}^1$ ramified
at six points, so can be given as $y^2=f$ with $f$ a polynomial of
degree $6$ with non-vanishing discriminant. Writing $f$ as a homogeneous
polynomial in two variables, say $f\in {\rm Sym}^6(W)$ with $W$
the $k$-vector space generated by $x_1,x_2$, and observing that we
may change the basis of $W$, we find a presentation of $\mathcal{M}_2$ as a
stack quotient
$$
\mathcal{M}_2 \sim [W_{6,-2}^0/{\rm GL}(W)],
$$
where we write $W_{a,b}$ for the ${\rm GL}(W)$-representation
 ${\rm Sym}^a(W) \otimes \det(W)^b$. Here the space $W_{6,-2}$
can be seen as the vector space of binary sextics $f$ with an action of
${\rm GL}(W)$ by
$$
f(x_1,x_2) \mapsto (ad-bc)^{-2} f(ax_1+bx_2,cx_1+dx_2) 
$$
for a matrix 
$\left( \begin{smallmatrix} a & b \\ c & d \\ \end{smallmatrix}\right)
\in {\rm GL}(2)$.
The subspace $W^0_{6,-2}$ of $W_{6,-2}$ is the space of $f$ with
non-vanishing discriminant.
The twisting by $\det(W)^{-2}$ is required to get the right stabilizer
for the generic $f$, namely $\pm {\rm Id}_W$.

This interpretation of $\mathcal{M}_2$ was used in \cite{CFG1} to
construct vector-valued Siegel modular forms of degree $2$ by using
invariant theory of binary sextics, thus extending and simplifying the description of
scalar-valued Siegel modular forms by invariants by Igusa \cite{Igusa1960,
Igusa1962}. Covariants define vector-valued modular forms and all Siegel
modular forms of degree $2$ on $\mathcal{A}_2$ can be constructed this way.
In \cite{CFG1} it was shown that the most basic covariant, the universal
binary sextic, defines a meromorphic Siegel modular form $\chi_{6,-2}$
of weight $(6,-2)$, that is,
it defines a meromorphic section of ${\rm Sym}^6({\EE}) \otimes
\det({\EE})^{-2}$ on $\mathcal{A}_2$. After multiplying $\chi_{6,-2}$
by Igusa's cusp form $\chi_{10}$ one obtains the holomorphic modular form
$\chi_{6,8}$, the `first' vector-valued Siegel modular cusp form of
degree $2$.

In the case of $g=3$, there is an analogous form $\chi_{4,0,8}$, a section of 
${\rm Sym}^4({\EE}) \otimes \det({\EE})^8$.
Here it derives from the description of the moduli space $\mathcal{M}_3^{nh}$
of non-hyperelliptic curves of genus three as a stack quotient 
$$
\mathcal{M}_3^{nh} \sim [W_{4,0,-1}^0/{\rm GL}(W)]\, ,
$$
where $W$ is now of dimension $3$ and $W_{4,0,-1}^0 \subset
{\rm Sym}^4(W) \otimes \det(W)^{-1}$ represents ternary quartics defining
smooth curves. In \cite{CFG2} this description
led to the construction of a meromorphic Teichm\"uller
modular form $\chi_{4,0,-1}$ of weight $(4,0,-1)$ and a (holomorphic)
Siegel modular form $\chi_{4,0,8}$ of degree $3$ and weight $(4,0,8)$.
Also in this case all Teichm\"uller and Siegel modular forms of genus $3$
on $\overline{\mathcal{M}}_3$ and $\mathcal{A}_3$ 
can be constructed from these forms by invariant theory.

This paper arises from the desire to construct these basic forms and similar
forms in a geometric way. 
We use cycle relations for effective divisors 
(or almost effective divisors) 
on the projectivized Hodge bundle to construct our forms. It is based on the
observation that an effective divisor $D$ on the projectivized Hodge bundle
${\PP}({\EE})$ with cycle class
$$
[D]=  [\mathcal{O}(j)] + k\, \lambda - \Delta 
$$
with positive integers $j, k$ and $\Delta$ an effective boundary class gives rise
to a section of ${\rm Sym}^j({\EE}) \otimes \det({\EE})^k$ vanishing on boundary divisors,
that is, a modular form.  This method produces
the basic modular forms $\chi_{6,8}$ and $\chi_{4,0,8}$ of degree $2$ and $3$ 
in an efficient way. 

This connection between divisors and modular forms can also be used in the other direction,
obtaining cycle classes for divisors on projectivized Hodge bundles. We give some
examples of this.

Another objective of this paper is to construct modular forms on
moduli spaces of hyperelliptic curves of genus $g$. 
For this we work with two descriptions of the moduli, a description as a stack quotient
and a description as a Hurwitz space. The latter space
$\mathcal{H}_{g,2}$ has as compactification  the space $\overline{\mathcal{H}}_{g,2}$
of admissible degree $2$ covers of genus $g$. 
In the stack description modular forms pull back to covariants
for the action of ${\rm GL}(2)$ on the space of binary forms of degree $2g+2$.

In the Hurwitz space description the relative canonical bundle 
of the universal curve over ${\mathcal H}_{g,2}$ 
can be viewed as the pull back of  $\mathcal{O}(g-1)$ from the trivial ${\mathbb P}^1$-bundle $P$  
over ${\mathcal H}_{g,2}$ equipped with $2g+2$ non-intersecting sections. 
Using the theory of admissible covers,  $P$ is compactified to a space $\overline{P}$, 
a fibration of rational stable curves with $2g+2$ marked points over  $\overline{\mathcal H}_{g,2}$,  
and we show that the line bundle $\mathcal{O}(g-1)$ on $P$ extends to a line bundle on $\overline{P}$ with the property
that its push down to
$\overline{\mathcal{H}}_{g,2}$ is close to the Hodge bundle. This allows us to construct
modular forms on $\overline{\mathcal{H}}_{g,2}$.

When we consider projectivized bundles projectivization is meant in the Grothendieck sense, so
that for a vector space $V$ the projective space ${\PP}(V)$
parametrizes hyperplanes in $V$. 

In an appendix we apply a method used in this paper to calculate the classes of certain divisors 
in the dual projectivized $k$-Hodge bundle that were determined by Gheorghita-Tarasca 
and by Korotkin-Sauvaget-Zograf.
\section*{Acknowledgements} We thank Fabien Cl\'ery, Carel Faber and Gavril Farkas and the referee for useful remarks.
\end{section}
\begin{section}{The Case of Genus Two}
Let $k$ be a field of characteristic not $2$. We consider the moduli space
$\mathcal{M}_2$ of curves of genus $2$ over $k$. This is a Deligne-Mumford stack
and it carries a universal curve $\pi: \mathcal{C}\to \mathcal{M}_2$ 
of genus $2$. 
The relative dualizing sheaf $\omega_{\pi}$ is base point free and 
thus defines a morphism $\varphi: \mathcal{C}\to {\PP}({\EE})$. 
For a curve $C$ the map $\varphi: C
\to {\PP}(\EE_{C})$ associates to a point the space of differentials vanishing
in that point.
We have a commutative diagram with $u$ the natural morphism
$$
\begin{xy}
\xymatrix{
\mathcal{C} \ar[rd]_{\pi} \ar[r]^{\varphi} & {\PP}({\EE})\ar[d]^u \\
& \mathcal{M}_2 \\
}
\end{xy}
$$
We let $\overline{\mathcal{M}}_2$ be the Deligne-Mumford compactification
and $\overline{\pi}: \overline{\mathcal{C}}\to \overline{\mathcal{M}}_2$ 
the corresponding universal curve. 
However, the extension
$\omega_{\overline{\pi}}$ of $\omega_{\pi}$ does not define an extension
of the morphism $\varphi$ to ${\PP}({\EE})$ 
over the boundary component $\Delta_1$ that parametrizes reducible curves. 

We consider 
the branch divisor $D \subset {\PP}({\EE})$ of the morphism $\varphi$. The divisor
$D$ is of relative degree $6$ in the ${\PP}^1$-bundle ${\PP}({\EE})$ 
over the base $\mathcal{M}_2$. We define $\overline{D}$
to be the closure of $D$ in ${\PP}({\EE})$ over $\overline{\mathcal{M}}_2$. 
In the rational Picard group of ${\PP}({\EE})$ we can write
$$
[\overline{D}]=[\mathcal{O}(6)]+u^*(A) 
$$
with $A$ a class in the rational Picard group of $\overline{\mathcal{M}}_2$
and $u: {\PP}({\EE}) \to \overline{\mathcal{M}}_2$ the natural projection.

We want to determine $A$ in terms of the generators $\lambda$, $\delta_0$ of the Picard group
of $\overline{\mathcal{M}}_2$.
We write $\lambda$ for the first Chern class of ${\EE}$ and $\delta_1$ (resp.\ 
$\delta_0$) for the
class of $\Delta_1$ (resp.\  $\Delta_0$) in the Picard group of the stack 
$\overline{\mathcal{M}}_2$; here $\Delta_0$ is the boundary component that
parametrizes irreducible curves with a double point.

In order to do this we extend the morphism $\varphi$.
It extends over a Zariski open part of $\Delta_0$ since $\omega_{\pi}$
has no base points there. However, over $\Delta_i$ with $i>0$ this system
has base points. We then use a base change as described in 
the appendix in Section \ref{app}. After a base change we have in an open 
neighborhood $U_i$ of the generic point of $\Delta_i$ 
a semi-stable family. If we take the base
to be $1$-dimensional  we get a semi-stable family
$\tilde{f}: \tilde{\mathcal{C}} \to \tilde{B}$
with as central fibre a chain $C'+R+C^{\prime\prime}$ with $R$ a $(-2)$-curve
and $C'$ and $C^{\prime\prime}$ of genus~$1$. The extension $\varphi'$
of $\varphi$ is given by $\omega_{\tilde{f}}(-R)$ with 
$\tilde{f}_*(\omega_{\tilde{f}}(-R))={\EE}_{\tilde{B}}$
and the morphism $\varphi': \tilde{C} \to {\PP}({\EE})$ contracts
$C'$ and $C^{\prime\prime}$ and is of degree $2$ on $R$. We refer
to the appendix, Section \ref{app} for the details.    
The morphism $\varphi'$ has the property that
$$
{\varphi'}^*(\mathcal{O}_{{\PP}({\EE}_{\tilde{B}})}(1))= \omega_{\tilde{f}}(-R)\, .
$$ 
\begin{proposition}\label{classofA} 
We have $[\overline{D}]=6\, [\mathcal{O}(1)]+u^*(8\, \lambda-\delta_0-\delta_1)$\, .
\end{proposition}
\begin{proof} We write $[\overline{D}]=6\, [\mathcal{O}(1)]+u^*(A)$.
We work with the above two types of $1$-dimensional 
families $f: C \to B$. 
The morphism $\varphi$ is ramified over $D$, and thus $\varphi^{\prime}$ is ramified
over $\overline{D}$ and contracts $C'$ and $C^{\prime\prime}$. We denote the
ramification divisor by $S$. We thus get (writing abusively line bundles and divisors
for the corresponding divisor classes)
$$
\omega_{\tilde{f}}={\varphi'}^*\omega_u + S + 2\, (C' +C^{\prime\prime}), \quad
{\varphi'}^* \overline{D}/2= S + 3\, (C'+C^{\prime\prime})\, ,
$$
where the first equation comes from adjunction $\omega_{\tilde{f}}+C'_{|C'}=\mathcal{O}_{C'}$ for
$C'$ and similarly for $C^{\prime\prime}$,
and the second one from 
$C' \cdot {\varphi'}^* \overline{D}=0=C^{\prime\prime} \cdot {\varphi'}^* \overline{D}$.
This gives
$$
\begin{aligned}
\omega_{\tilde{f}}& = 
{\varphi^{\prime}}^*(\omega_u+ \overline{D}/2) -(C'+C^{\prime\prime})\\
& =  {\varphi^{\prime}}^*(\mathcal{O}(-2)+u^*(\lambda) + 
\mathcal{O}(3) +u^*(A/2))- (C'+C^{\prime\prime}) \\
& = {\varphi^{\prime}}^*(\mathcal{O}(1)+u^*(\lambda+A/2))-(C'+C^{\prime\prime}) \\
& = \omega_{\tilde{f}}-R + \tilde{f}^*(\lambda+A/2) -(C'+C^{\prime\prime})\\
& = \omega_{\tilde{f}}+ \tilde{f}^*(\lambda+A/2-b_1)\, \\
\end{aligned}
$$
with $b_1$ the special point of $\tilde{B}$. This shows that $A=-2\, \lambda+2\, b_1$.
Because of the base change that 
we executed, we have $2\, b_1=\delta_1$ and we obtain
$A=-2\, \lambda+\delta_1$.
Now we use the well-known relation $10\, \lambda= \delta_0+2\, \delta_1$ 
(see \cite{Mumford}) and thus get
$A=-2\lambda+\delta_1=8\lambda -\delta_0-\delta_1$.
\end{proof}
\begin{remark} 
We indicate an alternative proof of this result in Remark \ref{remarkg=2}.
\end{remark}

An important remark is now that
$u_*(\mathcal{O}(1))={\EE}$ and $u_*(\mathcal{O}(m))={\rm Sym}^m({\EE})$ for $m\geq 1$. 
The divisor $\overline{D}$ 
with
$$
[\overline{D}]=[\mathcal{O}(6)]+u^*(8\lambda-\delta_0-\delta_1)
$$
is an effective divisor on ${\PP}({\EE})$.
We apply $u_*$ to the corresponding section~$1$ of $\mathcal{O}(\overline{D})$. 
By Proposition \ref{classofA}  we see that we get a regular section $\chi_{6,8}$ of
the vector bundle ${\rm Sym}^6({\EE})\otimes \det({\EE})^8$ over 
$\overline{\mathcal{M}}_2$. Moreover, this section
vanishes on the divisors $\Delta_0$ and $\Delta_1$. 
Note that the Torelli map extends to an 
isomorphism $\overline{\mathcal{M}}_2\cong \tilde{\mathcal{A}}_2$ with 
$\tilde{\mathcal{A}}_2$ the standard smooth compactification of
$\mathcal{A}_2$. Therefore our section defines a Siegel modular
form $\chi_{6,8}$ of weight $(6,8)$ that is a cusp form. 

\begin{corollary} Let $\overline{D}$ be the closure in ${\PP}({\EE})$ of the branch divisor of the canonical map
for the universal curve over $\mathcal{M}_2$. 
The push forward $u_*(s)$, 
with $s$ the natural section $1$ of $\mathcal{O}(\overline{D})$
on ${\PP}({\EE})$, defines a Siegel modular cusp 
form $\chi_{6,8}$ of degree $2$ and weight $(6,8)$.
\end{corollary}
The relation $10\, \lambda= \delta_0+2\, \delta_1$ quoted above implies that there exists 
a Siegel modular cusp form $\chi_{10}$ of degree $2$ and of
weight $10$ with divisor $\delta_0+2\, \delta_1$.
The quotient $\chi_{6,-2}:= \chi_{6,8}/\chi_{10}$ defines a meromorphic section
of ${\rm Sym}^6({\EE}) \otimes \det{\EE}^{-2}$ that is regular outside $\Delta_1$.

\smallskip

We now analyze the orders of vanishing along $\delta_1$ 
of $\chi_{6,8}$ and $\chi_{6,-2}$. When identifying $\overline{\mathcal{M}}_2$
with $\tilde{A}_2$ we also write $\mathcal{A}_{1,1}$ for $\delta_1$;
it is the locus of products of elliptic curves.

We analyze the orders by working locally on a family over a 
local base $B$ with central fibre a general point $b_1$ of the 
boundary divisor $\Delta_1$. As we mentioned before, the map 
$\varphi: {\mathcal C} \to {\PP}({\EE})$ defined over ${\mathcal M}_2$ 
does not extend to the whole $\overline{\mathcal C}$ over 
$\overline{\mathcal M}_2$ due to the fact that the canonical system has 
base points at the nodes of the curves over the boundary divisor $\Delta_1$. 
On the other hand, by the theory of admissible covers, the ramification 
divisor of the above map $\varphi $ extends to a divisor $S$ on  
$\overline{\mathcal C}$ in a way that avoids the above nodal locus. 
Namely, over $b_1 \in \Delta_1$ the fibre is a nodal curve $C$ which 
is the union of two elliptic curves $C_1$ and $C_2$ meeting at a point $p$. 
The restriction of the ramification divisor on each component is the union 
of the three --additional to $p$-- ramification points of the system $|\mathcal{O}(2p)|$. 
Therefore the extension of the map $\varphi $ is defined on the ramification
divisor $S$. 
The map $\varphi $ maps $C_1\backslash{\{p\}}$ and $C_2\backslash{\{p\}}$ 
to two distinct points $p_1$ and $p_2$ respectively which are defined as follows. 
The fibre of ${\mathbb P}({\mathbb E})$ over $b_1$ can be identified with 
${\PP}(H^0(C, \omega_C))$. The points of $H^0(C, \omega_C)$ have the form 
$(s_1,s_2)$, with $s_i$ a section of $H^0(C_i, \mathcal{O}(p))$. Then the point $p_1$ 
corresponds to the hyperplane $\{ (0, s_2), s_2\in H^0(C_2,\mathcal{O}(p))\}$ and the 
point $p_2$ corresponds to the hyperplane $\{ (s_1, 0), s_1\in H^0(C_1,\mathcal{O}(p))\}$.

The divisor $\overline{D}$, the image of the ramification divisor 
under the extended map $\varphi$, splits then into six irreducible 
components denoted by $D_1, \ldots, D_6$. Over our local base $B$ 
we thus have the six local sections $D_i$ ($i=1,\ldots, 6$) of the 
family ${\mathbb P}({\mathbb E}) \to B$. By the above description 
of the extension of the map $\varphi$, we may conclude that $D_1,D_2,D_3$ 
pass through $p_1$ and $D_4, D_5, D_6$  through $p_2$.

Lifting the sections $D_i$ locally to sections $\sigma_i$ of ${\EE}$ and choosing
a basis $e_1$, $e_2$ of ${\EE}$ over $B$ such that $e_1$ and $e_2$ determine $p_1$ and
$p_2$ in the fibre of ${\PP}({\EE})$ over $z=0$, 
we can write $\sigma_i=a_i e_1+b_i e_2$
for $i=1,\ldots,6$. Then at $z=0$ the functions $b_1,b_2,b_3$ and $a_4,a_5,a_6$ vanish,
while $a_1,a_2,a_3,b_4,b_5,b_6$ do not vanish.
Since by blowing up once we can separate, we may assume that these sections vanish
with order~$1$ at $z=0$. By construction the section $\chi$ of ${\rm Sym}^6({\EE})
\otimes \det({\EE})^{-2}$ is
locally given by
$$
\frac{\sigma_1 \cdots \sigma_6}{z} \, .
$$
We may write $\sigma_1 \cdots \sigma_6$ as
$$
\begin{aligned}
& a_1\cdots a_6 \, e_1^6 + (a_1a_2a_3a_4a_5b_6+\ldots + b_1a_2a_3a_4a_5a_6)\, e_1^5e_2 + \\
& (a_1a_2a_3a_4b_5b_6 +\ldots + b_1b_2a_3 a_4a_5a_6) \, e_1^4e_2^2 + \\
& (a_1a_2a_3b_4b_5b_6+ \ldots + b_1b_2b_3a_4a_5a_6)\, e_1^3e_2^3 + \ldots + b_1\cdots b_6 \, e_2^6\, .\\
\end{aligned}
$$
The order at $z=0$ of the coefficient of $e_1^ie_2^{6-i}$ equals
$$
\min_{\# \Lambda=i}(\#\Lambda^c \cap \{1,2,3\}+\# \Lambda \cap \{4,5,6\})
$$
with $\Lambda$ running though the subsets of $\{1,\ldots,6\}$ of cardinality $i$ and
$\Lambda^c$ denoting the complement.
We find for these orders $(3, 2, 1,0,1,2,3)$ for $i=0,\ldots,6$,
hence for the section $\chi$ given by $\sigma_1\cdots \sigma_6/z$ we find the orders $(2,1,0,-1,0,1,2)$.

\begin{corollary} The section $1$ of the line bundle $\mathcal{O}(\overline{D})$ on ${\PP}({\EE})$ over
$\overline{\mathcal{M}}_{2}$ pushes down 
via ${\PP}({\EE})\to \overline{\mathcal{M}}_2$ 
to the
meromorphic modular form $\chi_{6,-2}$ on $\overline{\mathcal{M}}_2=\tilde{\mathcal{A}}_2$.
The orders of the seven coordinates of  $\chi_{6,-2}$ along $\mathcal{A}_{1,1}$ in $\mathcal{A}_2$
are $(2,1,0,-1,0,1,2)$. 
\end{corollary}
These orders are in agreement with the result of \cite{CFG1} where $\chi_{6,-2}$ was constructed by
invariant theory and properties of modular forms were used to determine these orders.

\smallskip
A different way to construct the form $\chi_{6,8}$ uses the so-called Weierstrass
divisor $W$ in the dual bundle:
$$
W:=\{ (C,\eta) \in {\PP}({\EE}^{\vee}): \text{\rm ${\rm div}(\eta)$ contains a 
Weierstrass point} \}
$$
over $\mathcal{M}_2$.  
Here $C$ denotes a curve of genus $2$ and $\eta$ a differential form on $C$. 
We let $\overline{W}$ 
be the closure of $W$ over $\overline{\mathcal{M}}_2$. 
We then have an identity
due to Gheorghita \cite[Thm 1]{Gheorghita}
$$ 
[\overline{W}] = 6\, [\mathcal{O}_{{\PP}({\EE}^{\vee})}(1)] + 34\, \lambda -3\, \delta_0 - 5\, \delta_1\, ,
$$
where we write $\lambda$ and $\delta_i$ for the pullback of $\lambda$ and $\delta_i$ 
to ${\PP}({\EE}^{\vee})$. Now $\overline{W}$ is an effective divisor and
the push forward of the section $1$ of $\mathcal{O}(\overline{W})$ is a section of
${\rm Sym}^6({\EE}^{\vee})\otimes \det({\EE})^{34}\otimes \mathcal{O}(- 3\delta_0 - 5 \delta_1)$.
For $g=2$ we have ${\EE}^{\vee}\cong {\EE}\otimes
\det({\EE})^{-1}$, hence
${\rm Sym}^6({\EE}^{\vee})\cong {\rm Sym}^6({\EE})\otimes \det({\EE})^{-6}$. This implies that
under the isomorphism of ${\PP}^1$-bundles ${\PP}({\EE})\cong 
{\PP}({\EE}^{\vee})$ the isomorphism identifies $[\overline{W}]$ with $[\overline{D}]$, and we
get in the dual bundle
$$
[\overline{W}]= 6\, [\mathcal{O}(1)] + 28 \, \lambda -3\, \delta_0-5\, \delta_1 = 
6\, [\mathcal{O}(1)]+8\, \lambda-\delta_0-\delta_1\, . 
$$
Using push forward we find again a form of weight $(6,8)$ vanishing on $\delta_1$ and $\delta_0$.
Up to a multiplicative non-zero constant this is $\chi_{6,8}$. 
\begin{remark}
The identity $[\overline{W}]= 6\, [\mathcal{O}(1)] + 28 \, \lambda -3 \, \delta_0-5\, \delta_1 $
implies that there exists a modular form of weight $(6,28)$ vanishing with multiplicity $3$ on $\Delta_0$
and multiplicity $5$ on $\Delta_1$, but this is (up to a multiplicative constant)
the form
 $\chi_{10}^2 \, \chi_{6,8}$ with $\chi_{10}$ the form of weight $10$ with divisor
$\delta_0+2\delta_1$ that displays the relation $10\, \lambda=\delta_0+2\delta_1$.
\end{remark}

\end{section}
\begin{section}{The class of the $k$-canonically embedded curve}\label{classkcanonical}
For the calculation of the classes of effective divisors in ${\PP}({\EE})$
related to  the canonical image of the universal curve it is helpful
to have (part of) the class of the closure of the canonical image in ${\PP}({\EE})$
over $\overline{\mathcal{M}}_g$. Without extra effort we can and will extend the calculation
to the case of the $k$-canonically embedded curve for $k \geq 1$.

We consider the universal family $\pi: \overline{\mathcal{C}}_g\to 
\overline{\mathcal{M}}_g$. This comes with a natural vector bundle
${\EE}_k=\pi_*(\omega_{\pi}^{\otimes k})$ for $k\in{\ZZ}_{\geq 1}$
and for $k=1$ this is the Hodge bundle ${\EE}_1={\EE}$. We write
$u: {\PP}({\EE}_k)\to \overline{\mathcal{M}}_g$ for the natural map.
For $k\geq 2$ the sheaf $\omega_{\pi}^{\otimes k}$ is base point free
for stable curves and the surjection 
$\pi^*{\EE}_k \to \omega_{\pi}^{\otimes k}$
defines a morphism $\varphi_k: \overline{\mathcal{C}}_g \to {\PP}({\EE}_k)$. 
For $k=1$ the sheaf $\omega_{\pi}$ is base point free on $\mathcal{M}_g \cup \Delta_0^0$, 
with $\Delta_0^0\subset \Delta_0$ the open locus with only disconnecting nodes, but it has base points on the nodes lying over
the generic points of the boundary components $\Delta_i$ for $i>0$.
The appendix Section \ref{app} describes the closure of the
image over an open neighborhood of the generic point of $\Delta_i$.

We denote by $\Gamma_k$ the image of $\varphi_k(\mathcal{C}_g)$ over
$\mathcal{M}_g\cup \Delta_0^0$ and
by $\overline{\Gamma}_k$ the closure of $\Gamma_k$ in ${\PP}({\EE}_k)$ over
$\overline{\mathcal{M}}_g$.
We can write the class of $\Gamma_k$ as a cycle on ${\PP}({\EE}_k)$ 
as
$$
[\overline{\Gamma}_k]= \sum_{i=0}^{r-2} h_k^i u^*(\beta_{r-2-i}), \eqno(1)
$$
with $r={\rm rank}({\EE}_k)$, $h_k=c_1(\mathcal{O}_{{\PP}({\EE}_k)}(1))$
and $\beta_j$ a codimension $j$ cycle on $\overline{\mathcal{M}}_g$.
\begin{proposition}\label{curveclass}
We have
$ \beta_0=2\, k (g-1)$ and  $\beta_1=k^2\kappa_1-2k(g-1)c_1({\EE}_k)-\epsilon$
with $\epsilon=\sum_{i=1}^{[g/2]}\delta_i$ for $k=1$ and $\epsilon=0$ else.
\end{proposition}
Using $c_1({\EE}_k)=\frac{k(k-1)}{2}\kappa_1+\lambda$ and 
$\kappa_1=12\lambda - \sum_{i=0}^{[g/2]}\delta_i$ (by \cite{Mumford1977}) 
we can write $\beta_1$ as
$$
\beta_1=\left( (1-g)(12k^3-12k^2+2k)+ 12\,k^2 \right)\, \lambda + 
\left( (g-1)k-g\right)\, k^2  \sum_{i=0}^{[g/2]} \delta_i - \epsilon \, .
$$
\begin{proof}
We start with the case $k\geq 2$. In this case $\overline{\Gamma}_k$
is the image of $\overline{\mathcal{C}}_g$ under $\varphi_k$
and $\varphi_k^*(h_k)=k\, \omega_{\pi}$.
Since the image of the generic fibre has degree $2k(g-1)$ we find
$\beta_0=2k(g-1)$. The hyperplane class $h_k$ satisfies
$\sum_{i=0}^r (-1)^i h_k^i c_{r-i}({\EE}_k)=0$. 
For dimension reasons we have $u_*(h_k^i)=0$ for $i\leq r-2$ and $u_*(h_k^{r-1})=1$.
Thus we get
$$
u_*({\varphi_k}_*[1]h_k^2)=u_*{\varphi_k}_*(k^2\omega_{\pi}^2)=
k^2\pi_*(\omega_{\pi}^2)=k^2\,\kappa_1\, ,
$$
with $\pi_*(\omega_{\pi}^2)=\kappa_1$, while on the other hand by (1)
$$
u_*(({\varphi_k}_*[1])h_k^2)=(u_*h_k^r)\beta_0+(u_*h_k^{r-1})\beta_1=
2k(g-1)\, c_1({\EE}_k)+\beta_1\, ,
$$
giving $\beta_1=k^2\kappa_1-2k(g-1)c_1({\EE}_k)$
and this settles the case
$k\geq 2$. The same argument works for $k=1$ as long as we work on
$\mathcal{M}_g\cup \Delta_0^0$. To get the coefficients of the $\delta_i$
for $i>0$ we work over a $1$-dimensional base $B$ in an open neighborhood $U_i$ of the generic point of $\Delta_i$ with special fibre
$C'+R+C^{\prime\prime}$ as in the appendix
Section \ref{app} where the extension $\varphi'$ of $\varphi$ is
defined by $\omega_{\pi'}(-R)$. The contribution of $\delta_i$ to
$2\beta_1$ is now $\pi^{\prime}_*(\omega(-R)^2)$, where the coefficient $2$ of
$\beta_1$ comes from the fact that $\varphi'$ is of degree $2$ on $R$.
We get 
$\pi^{\prime}_*(\omega_{\pi'}(-R)^2)=
\pi^{\prime}_*(\omega_{\pi'}^2)+\pi^{\prime}_*(R^2)=
-2\, \delta_i-2\, \delta_i=-4\, \delta_i$, 
as the fibre has two singular points and $R^2=-2$. 
Putting everything together results in the
given formula.
\end{proof}
\end{section}
\begin{section}{The Case of Genus Three}
Here there is no restriction on the characteristic. 
We consider the moduli stack
$\mathcal{M}_3$ of curves of genus $3$ over our field $k$
and the universal
curve $\pi:\mathcal{C}\to \mathcal{M}_3$.
The canonical map defines a morphism $\varphi: \mathcal{C} \to {\PP}({\EE})$
and we thus obtain the image divisor $D$ in ${\PP}({\EE})$ over $\mathcal{M}_3$. 
We have a commutative diagram 
$$
\begin{xy}
\xymatrix{
\mathcal{C} \ar[rd]_{\pi} \ar[r]^{\varphi} & {\PP}({\EE})\ar[d]^u \\
& \mathcal{M}_3 \\
}
\end{xy}
$$
We consider the closure $\overline{D}$ of $D$ in ${\PP}({\EE})$ 
over $\overline{\mathcal{M}}_3$. 
The canonical image of the generic curve is a quartic curve.
Thus we have a relation $ [\overline{D}]=[\mathcal{O}(4)]+u^*(A)$
in the rational Picard group of ${\PP}({\EE})$
with $A$ a divisor class on $\overline{\mathcal{M}}_3$ 
given in Proposition \ref{curveclass}. 

\begin{corollary}\label{classofAII}
We have 
$ [\overline{D}]=[\mathcal{O}(4)]+u^*(8\, \lambda-\delta_0-2\, \delta_1)$.
\end{corollary}

The divisor $\overline{D}$ is effective over $\overline{\mathcal{M}}_3-\Delta_1$. 
Because of the relation 
$[\overline{D}]=[\mathcal{O}(4)]+u^*(8\, \lambda-\delta_0-2\, \delta_1)
$
the corresponding section $1$ of $\mathcal{O}(\overline{D})$ maps under $u_*$ to a section $\psi$ of
${\rm Sym}^4({\EE}) \otimes \det({\EE})^{8}$ that is regular outside $\Delta_1$
and vanishes on $\Delta_0$. In view of the even powers ${\rm Sym}^4$ and $8$, 
this section $\psi$ is invariant under the action of $-1$
on the fibres of ${\EE}$. As the action of $-1$ defines the involution of the 
double covering
of stacks $\mathcal{M}_3 \to \mathcal{A}_3$, the section $\psi$ 
descends to a section $\chi_{4,0,8}$ of 
${\rm Sym}^4({\EE}) \otimes \det({\EE})^{8}$ on the image of 
$\overline{\mathcal{M}}_3-\Delta_1$ under the Torelli morphism $\overline{\mathcal{M}}_3
\to \tilde{\mathcal{A}}_3$, with $\tilde{\mathcal{A}}_3$
the standard second Voronoi compactification of $\mathcal{A}_3$.
Since the image of $\Delta_1$ in $\tilde{\mathcal{A}}_3$ is of codimension~$2$,
the section $\chi_{4,0,8}$ extends to a regular section of 
${\rm Sym}^4({\EE}) \otimes \det({\EE})^{8}$ on all of $\mathcal{A}_3$, 
and then by the Koecher Principle it extends to
$\tilde{\mathcal{A}}_3$. Thus it defines a regular Siegel
modular cusp form $\chi_{4,0,8}$ of degree $3$ and weight $(4,0,8)$.

\begin{corollary} Let $\overline{D}$ be the closure
of the canonical curve over $\mathcal{M}_3$ in ${\PP}({\EE})$ and
$s$ the natural section $1$ of $\mathcal{O}(\overline{D})$. Then $\chi_{4,0,8}=u_*(s)$ is a
Teichm\"uller modular form and it descends to a Siegel modular cusp form
of degree $3$ and weight $(4,0,8)$.
\end{corollary}

The class $\mathfrak{H}$ of the hyperelliptic locus $\overline{\mathcal{H}}_3$
in $\overline{\mathcal{M}}_3$ satisfies (\cite[p.\ 140]{Harris})
$$
\mathfrak{h}= 9\, \lambda -\delta_0 - 3\, \delta_1 \, . \eqno(2)
$$
The relation (2) shows that there exists a scalar-valued Teichm\"uller modular
form $\chi_9$ of weight $9$ on $\overline{\mathcal{M}}_3$. Its square is invariant under
the action of $-1$ on the fibres of ${\EE}$, hence descends to a Siegel modular
form of weight $18$. Up to a multiplicative scalar this is Igusa's modular form
$\chi_{18}$.

If we divide $\chi_{4,0,8}$ by $\chi_9$ we obtain a meromorphic section of 
${\rm Sym}^4({\EE}) \otimes \det({\EE})^{-1}$ on $\overline{\mathcal{M}}_3$
that is regular on $\mathcal{M}_3$ outside the hyperelliptic locus. This form was used in \cite{CFG2}
to construct Teichm\"uller modular forms and Siegel modular forms by invariant
theory.
\end{section}
\begin{section}{Moduli of hyperelliptic curves as a stack quotient} \label{hyperell-stack}
In this section we discuss the stack quotient description of 
the moduli of hyperelliptic curves.
We consider
hyperelliptic curves in characteristic not $2$. A hyperelliptic curve of genus $g$ is
a morphism $\alpha: C \to {\PP}^1$ of degree $2$ where $C$ is a smooth curve of genus $g$.
A morphism $a: \alpha \to \alpha'$ between two hyperelliptic curves is a commutative diagram
$$
\begin{xy}
\xymatrix{
{C} \ar[r]\ar[d]_{\alpha}& C'\ar[d]^{\alpha'} \\
{\PP}^1 \ar[r] & {\PP}^1 \\
}
\end{xy}
$$
A hyperelliptic curve $C$ of genus $g$ can be written as $y^2=f(x)$ with $f \in {\CC}[x]$
of degree $2g+2$. In fact, choosing a basis $(x_1,x_2)$ of the $g^1_2$ defines the
morphism $\alpha$. Let $W=\langle x_1,x_2 \rangle$, a vector space (over our algebraically
closed base field) of dimension
$2$, and $L=\alpha^*(O_{\PP^1}(1))$. By Riemann-Roch we have $\dim H^0(C,L^{g+1})=g+3$,
while $\dim {\rm Sym}^{g+1}(W)=g+2$, so we have a non-zero element $y\in H^0(C,L^{g+1})$ which is
anti-invariant under the involution corresponding to $\alpha$. The anti-invariant subspace
of $H^0(C,L^{g+1})$ has dimension $1$. Then $y^2$ is invariant
and lies in ${\rm Sym}^{2g+2}(W)$. Thus we find the equation $y^2=f(x_1,x_2)$
with $f$ homogeneous of degree $2g+2$ and with non-zero discriminant.

We have made two choices here: a generator  $y$ of $H^0(C,L^{g+1})^{(-1)}$,
a space of dimension~$1$, and a basis of $W$.
We can change the choice of $y$ (by a non-zero scalar)
and the choice of a basis of $W$ by $\gamma=(a,b;c,d) \in {\rm GL}(W)$. 
The action of  ${\rm GL}(W)$ is on the right via
$$
f(x_1,x_2) \mapsto f(ax_1+bx_2,cx_1+dx_2) \, .
$$
If we let ${\rm GL}(W)$ act on $y$ by a power of the determinant,
then this action preserves the type of equation.
In inhomogeneous form the action by ${\rm GL}(W)$ is by
$$
f\mapsto f(\frac{ax+b}{cx+d}), \quad y\mapsto y/(cx+d)^{g+1} \, ,
$$
with the following effect on the equation:
$$
y^2=f(x) \mapsto y^2=(cx+d)^{2g+2}f(\frac{ax+b}{cx+d}) \, .
$$
The last expression on the right-hand-side
can be written as binary form of degree $2g+2$.

The stabilizer of a generic $f \in {\rm Sym}^{2g+2}(W)$ is $\mu_{2g+2}$, the roots of unity of order dividing $2g+2$.
Since we want a stabilizer of order $2$ for the generic element, we consider a twisted action:
define the ${\rm GL}(W)$-representation
$$
W_{a,b}={\rm Sym}^a(W) \otimes {\det}(W)^{\otimes b}\, .
$$
This can be identified with ${\rm Sym}^a(W)$ as a vector space, but the action by ${\rm GL}(W)$
is different.
Inside this space $W_{a,b}$ we have the open subspace
$W_{a,b}^0$ of homogeneous polynomials of degree $a$
with non-zero discriminant.
We now distinguish two cases.

\textbf{Case 1.} $g$ even. Here we consider the stack quotient
$$
[W_{2g+2,-g}^0/{\rm GL}(W)] \, .
$$
This stack quotient can be identified with the moduli stack $\mathcal{H}_g$ of hyperelliptic curves of genus $g$ for $g$ even.
Indeed, the action of $t\cdot {\rm Id}_W$ is (on inhomogeneous equations) by
$$
f \mapsto t^{-2g} f, \quad y \mapsto y/t^{g+1}\, ,
$$
hence $y^2=f$ maps to $y^2=t^2\, f$, so that the stabilizer is $\mu_2$, as required.
Note also that the action of $-1 \in {\rm GL}(W)$ is by $y \mapsto -y$, so $y$ is an
odd element.
A basis of $H^0(C,\Omega^1_C)$ is given by
$$
x^i dx/y, \quad (i=0,\ldots,g-1)\, .
$$
The action on $dx$ is by $(ad-bc) dx/(cx+d)^2$ resulting in the action on the space of
differentials by
$$
x^i dx/y \mapsto (ad-bc) (cx+d)^{g-1-i} (ax+b)^i\, dx/y \, .
$$
If we forget the twisted action on $y$,
we can identify $H^0(C,\Omega^1_C)$ with $W_{g-1,1}$. But  $y^2$ must
be viewed as an element of $W_{2g+2,-g}$,
so the action of $t \, 1_W$
on $y$ should be twisted by  $t^{-g}={\det}^{-g/2}$. We get
$$
H^0(C,\Omega_C^1) \cong W_{g-1,(2-g)/2} 
\quad \text{\rm  for $g$ even.}
$$
We see that under the identification $h: [W_{g+2,-g}^0/{\rm GL}(W)] \langepijl{\sim} \mathcal{H}_g$
the pullback $h^*({\EE})$ 
of the Hodge bundle ${\EE}$
is the equivariant bundle $W_{g-1,(2-g)/2}$.
The action of $-1_W$ is by $-1$ on $W_{g-1,(2-g)/2}$.
We also observe
$h^*(\det({\EE}))={\det}(W)^{g/2}$.

\bigskip
\textbf{Case 2.} $g$ odd. Here we take $W_{2g+2,-g+1}$.
\begin{remark} If we consider $W_{2g+2,r}$ then $r$ has to be even,
since as above
we later view $y^2$ as an element of $W_{2g+2,r}$
and we need an action by ${\det}^{r/2}$
on $y$.
\end{remark}
Here the stabilizer of a generic element is $\mu_4$. Now on inhomogeneous equations the action is
by
$$
f \mapsto t^{-2g+2} \, f, \quad y \mapsto y/t^{g+1}\, ,
$$
hence $y^2=f$ maps to $y^2=t^4 \, f$. Note that here $-1_W$
acts by $f\mapsto f$ and $y\mapsto y$.
But $\sqrt{-1}_{W}$ acts by $f\mapsto f$ and $y\mapsto -y$. 
To get the right stack
quotient with stabilizer of the generic element of order $2$, we take
$$
[W^0_{2g+2,1-g}/({\rm GL}(W)/(\pm 1_W))]\, .
$$
The action on the differentials $x^idx/y$ with $i=0,\ldots,g-1$ is by
$$
x^idx/y \mapsto (ad-bc)^{(1-g)/2} (cx+d)^{g-1-i}(ax+b)^i\, dx/y\, ,
$$
hence without twisting we get $H^0(C,\Omega_C^1)=W_{g,1}$. Since we
view $y^2$ as element of $W_{g+3,1-g}$
we find under $h: [W^0_{2g+2,1-g}/({\rm GL}(W)/(\pm 1_W))] \langepijl{\sim} \mathcal{H}_g$
that
$h^*({\EE})=W_{g-1,(3-g)/2}$.
The element $\sqrt{-1}\, 1_W$ acts on the differentials as $(-1)^{(3-g)/2}(\sqrt{-1})^{g-1}
=-1$.

We summarize.

\begin{proposition} Writing
$W_{a,b}={\rm Sym}^{a}(W)\otimes \det(W)^{b}$
we have the identification of stacks
$$
h^{-1}: \mathcal{H}_g \langepijl{\sim} \begin{cases}
[W_{2g+2,-g}^0/{\rm GL}(W)] & g \quad  \text{\rm even} \\
[W^0_{2g+2,1-g}/({\rm GL}(W)/(\pm 1_W))] & g \quad \text{\rm odd,}\\
\end{cases}  
$$
and
$$
h^*({\EE})\cong \begin{cases} W_{g-1,(2-g)/2}  &  g \quad  \text{\rm even} \\
W_{g-1,(3-g)/2} & g \quad \text{\rm odd,}\\
\end{cases} 
\qquad
h^*(\det({\EE}))= \begin{cases} \det(W)^{g/2} & g \quad  \text{\rm even} \\
\det(W)^{g} & g \quad \text{\rm odd.} \\ \end{cases} 
$$
\end{proposition}
For a somewhat different description see \cite[Cor.\ 4.7, p.\ 654]{A-V}.

\smallskip

Recall that the moduli stack $\mathcal{H}_g$ has as compactification the closure
$\overline{\mathcal{H}}_g$ of $\mathcal{H}_g$ inside the moduli stack $\overline{\mathcal{M}}_g$.
The Picard group of $\mathcal{H}_g$ is known by \cite{A-V} to be finite cyclic for $g\geq 2$ 
of order $4g+2$ if $g$ is even and $8g+4$ else. The rational Picard group of
$\overline{\mathcal{H}}_g$ is known by Cornalba (see \cite{Cornalba}) 
to be free abelian of rank $g$ generated by
classes $\delta_i$ and $\zeta_j$ for $i=0,\ldots,\lfloor g/2\rfloor$ and
$j=1,\ldots \lfloor(g-1)/2\rfloor$. Cornalba gives also the first Chern class
$\lambda$ of the Hodge bundle ${\EE}$ on $\overline{\mathcal{H}}_g$
$$
(8g+4)\, \lambda= g \, \delta_0+ 4 \sum_{i=1}^{\lfloor g/2 \rfloor} i(g-i)\, \delta_i
+ 2 \, \sum_{i=1}^{\lfloor (g-1)/2 \rfloor} (i+1)(g-i)\, \zeta_i \, ,
$$
where the generic point of the divisor $\zeta_i$ has an admissible model $C' \cup C^{\prime\prime}$ 
with two nodes $C' \cap C^{\prime\prime}=\{ p,q\}$ 
mapping to a union of two ${\PP}^1$, with $2i+2$ marked
points on $C'$, see Figure 1
in Section \ref{Hurwitz}.
\end{section}
\begin{section}{Modular forms on the hyperelliptic locus of genus three} \label{hyperellg3}
Let ${\EE}$ be the Hodge bundle on $\overline{\mathcal{H}}_3$. 
By a modular form of weight $k$ on $\overline{\mathcal{H}}_3$ 
we mean a section of $\det({\EE})^{\otimes k}$.
The construction in the preceding section shows that a modular form of weight $k$
on $\overline{\mathcal{H}}_3$ when pulled back to the
stack $[W^0_{2,0}/({\rm GL}(W)/\pm{\rm id}_W)]$ gives rise to an
invariant of degree $3k/2$. Indeed, it defines a section of the equivariant
bundle $\det(W)^{3k}$ invariant under ${\rm SL}(W)$, but
in view of the fact that we divide
by the action of ${\rm GL}(W)/(\pm {\rm id}_W)$ this yields an invariant of degree $3k/2$.

Let $M_k(\Gamma_3)=H^0(\mathcal{A}_3,\det({\EE})^k)$ be the space of Siegel modular
forms of degree $3$ on $\Gamma_3={\rm Sp}(6,{\ZZ})$.
In \cite{Igusa1967} Igusa considered an exact sequence
$$
0 \to M_{k-18}(\Gamma_3) \langepijl{\cdot \chi_{18}} M_k(\Gamma_3) \to I_{3k/2}(2,8) 
$$
with $I_{d}(2,8)$ the vector space of invariants of degree $d$ of binary octics.
We can interpret Igusa's sequence in the following way. A Siegel modular form of weight $k$
defines by restriction to the hyperelliptic locus a modular form of weight $k$
on $\overline{\mathcal{H}}_3$ and it thus defines an invariant of degree $3k/2$.

For each irreducible representation $\rho$ of ${\rm GL}(3)$
we have a vector bundle ${\EE}_{\rho}$ made from ${\EE}$ by a Schur functor.
By a modular form of weight $\rho$ on  $\overline{\mathcal{H}}_3$ we mean
a section of a vector bundle ${\EE}_{\rho}$.
We can pull back to the stack $[W^0_{8,-2}/({\rm GL}(W)/\pm{\rm id}_W)]$,
but the situation is more involved as ${\rm Sym}^n({\rm Sym}^2(W))$
decomposes as a representation of ${\rm GL}(W)$. For example, we have
with $W_{a,b}={\rm Sym}^a(W)\otimes \det(W)^b$
$$
h^*({\rm Sym}^4({\EE}))= {\rm Sym}^4({\rm Sym}^2(W))=W_{8,0}\oplus W_{4,2} \oplus W_{0,4}\, .
$$
Here and in the rest of this section we assume that the characteristic is $0$, or not $2$,
and high enough for the representation theory (plethysm) to work\footnote{Alternatively one could 
use divided powers as in \cite[3.1]{AFPRW}}.
In this case we can consider the restriction of the Siegel modular form $\chi_{4,0,8}$ to the hyperelliptic
locus and we know that it does not vanish identically by \cite[Lemma 7.7]{CFG2}. 
On the other hand we have the basic covariant
$f_{8,-2}$, the diagonal section of $W_{8,-2}$ over the stack $[W_{8,-2}^0/({\rm GL}(W)/\pm {\rm id}_W)]$.

The discriminant form $\mathfrak{d}$ of binary octics, an invariant of degree $14$, 
does not define a modular form, but its third power $\mathfrak{d}^3$ does. It defines a modular
form of weight $28$, see \cite[p.\ 811]{Tsuyumine1986} and also Remark \ref{weight14}.

Via the projection $p_{8,0}: {\rm Sym}^4({\rm Sym}^2(W))\to W_{8,0}$ a section of ${\rm Sym}^4({\EE}) \otimes
\det({\EE})^8$ defines a covariant of bi-degree $(8,24/2)=(8,12)$ for the action of ${\rm GL}(W)$.

\begin{proposition}\label{binary_octic}
The restriction to the hyperelliptic locus of the section $\chi_{4,0,8}$ corresponds via the projection
$p_{8,0}$ to a multiple of the covariant $f_{8,-2}\cdot \mathfrak{d}$ with
$\mathfrak{d}$ the discriminant of binary octics.
\end{proposition}
\begin{proof}
By restricting and projecting we obtain a covariant of bi-degree $(8,12)$. This covariant is divisible
by the discriminant and does not vanish on the locus of smooth hyperelliptic curves. Therefore,
division by $\mathfrak{d}$ 
gives a non-vanishing covariant of bidegree $(8,-2)$. Taking into account the
`twisting' by $\det(W)^{-2}$, this must be a multiple of the universal binary octic.
\end{proof}

We will discuss the other two projections later in Lemma \ref{3summands}.
Note that the divisor $D$, the canonical image of the universal curve in ${\PP}({\EE})$
that defines $\chi_{4,0,8}$, has a restriction to the locus of smooth hyperelliptic curves
which is divisible by $2$. Indeed, the canonical image of a hyperelliptic curve is a double
conic. This suggests that we can take the `square root' of the restriction of $\chi_{4,0,8}$
to the hyperelliptic locus. However, the boundary divisors prevent this. If we take a level
cover of the moduli space we can construct a modular form of weight $(2,0,4)$. 
We will carry this out later (in Corollary \ref{chi204}), working on a Hurwitz space
that we shall introduce in the next section.
\end{section}
\begin{section}{The Hurwitz space of admissible covers of degree two} \label{Hurwitz}
In this and the following sections will use the other description of the moduli
of hyperelliptic curves, namely the moduli space $\overline{\mathcal{H}}_{g,2}$ of admissible covers
of degree $2$ and genus $g$ in the sense of \cite{HMu}, see \cite{HMo}. Thus we are
looking at covers
$f: C \to P$ of degree $2$ with $C$ nodal
of genus $g$ and $P$ a stable $b$-pointed curve of genus $0$.
Here the $b=2g+2$ branch points are ordered and $\mathcal{H}_{g,2}\to \mathcal{H}_g$
is a Galois cover with Galois group the symmetric group $\mathfrak{S}_{2g+2}$.

The boundary $\overline{\mathcal{H}}_{g,2}-\mathcal{H}_{g,2}$ consists of finitely
many divisors that we shall denote by $\Delta_b^{\Lambda}=\Delta^{\Lambda}$, where we omit
the index $b$ if $g$ is clear. Here the index $\Lambda$ defines
a partition
$ \{1,2,\ldots,b\}=\Lambda\sqcup \Lambda^c$,
and the generic point of $\Delta^{\Lambda}$ corresponds to an admissible
cover that maps to a stable curve of genus $0$ that is
the union of two copies of ${\PP}^1$, one containing the
points with mark in $\Lambda$, the other one those with mark in $\Lambda^c$.
Here we will assume that $\# \Lambda=j$ with $2 \leq j \leq g+1$.

The parity of $\# \Lambda$ plays an important role here.
If $\# \Lambda=2i+2$ is even, then the generic admissible cover corresponding
to a point of $\Delta^{\Lambda}$ is a union  $C_i \cup C_{g-i-1}$ that is a 
double cover of a union of
two rational curves ${\PP}_1$ and ${\PP}_2$ with $C_i$ lying over ${\PP}_1$ and
$C_{g-i-1}$ over ${\PP}_2$.
Here $C_i$ (resp.\ $C_{g-i-1}$) has genus $i$ (resp.\ $g-i-1$) with
$0 \leq i \leq (g-1)/2$ and is ramified over the points of $\Lambda$
(resp.\ $\Lambda^c$).

\begin{center}
\begin{pspicture}(-3,-2)(3,3)
\psline[linecolor=red](-0.6,-1.2)(3,0) 
\psline[linecolor=blue](-3,0)(0.6,-1.2) 
\pscurve[linecolor=blue](-3,1.8)(0,1)(0.5,0.4)(-0.6,0) 
\pscurve[linecolor=red](3,1.8)(0,1)(-0.5,0.4)(0.6,0) 
\psline{->}(0,-0.1)(0,-0.6)
\rput(-3.5,2){$C_i$}
\rput(3.5,2){$C_{g-i-1}$}
\rput(-3.5,0){${\PP}_1$}
\rput(3.5,0){${\PP}_2$}
\rput(0,1.3){$p$}
\rput(0,0.3){$q$}
\rput(0,-1.2){$s$}
\rput(0,-2){\tt Fig.\ 1: $\Lambda$ even}
\end{pspicture}
\end{center}

If $\# \Lambda=2i+1$ is odd with $1\leq i \leq g/2$, then we have a union 
$C_i \cup C_{g-i}$ lying over ${\PP}_1 \cup {\PP}_2$, where $C_i$ (resp.\
$C_{g-i}$) of genus $i$ (resp.\ $g-i$) is ramified over $\Lambda$ and in $p$,
the intersection of $C_i$ and $C_{g-i}$, (resp.\
over $\Lambda^c$ and in $p$). Note that $p$ is a simple node.

\begin{center}
\begin{pspicture}(-3,-2)(3,3)
\psline[linecolor=red](-0.6,-1.2)(3,0) 
\psline[linecolor=blue](-3,0)(0.6,-1.2) 
\pscurve[linecolor=blue](-3,1.8)(0,1)(-1,-0.3) 
\pscurve[linecolor=red](3,1.8)(0,1)(1,0.9) 
\psline{->}(0,-0.1)(0,-0.6)
\rput(-3.5,2){$C_i$}
\rput(3.5,2){$C_{g-i}$}
\rput(-3.5,0){${\PP}_1$}
\rput(3.5,0){${\PP}_2$}
\rput(-0.4,1.0){$p$}
\rput(0,-1.2){$s$}
\rput(0,-2){\tt Fig.\ 2: $\Lambda$ odd}
\end{pspicture}
\end{center}

Assuming that $g$ and $b=2g+2$ are fixed we will write
$$
\Delta_{j}=\sum_{\# \Lambda=j} 
\Delta^{\Lambda} \qquad \text{\rm with $2\leq j < b/2$}
$$
and provide the symmetric case with a factor $1/2$, that is,
$\Delta_{b/2}= \frac{1}{2} \sum_{\#\Lambda=b/2} \Delta^{\Lambda}$.

\end{section}
\begin{section}{Divisors on the moduli of stable curves of genus zero}
For later use we recall some notation and facts concerning divisors on
the moduli spaces $\overline{\mathcal{M}}_{0,n}$. We refer to \cite{Ke}.
The boundary divisors on  $\overline{\mathcal{M}}_{0,n}$ are denoted by
$S_n^{\Lambda}$ and are indexed by partitions $\{1,\ldots,n\}=\Lambda\sqcup \Lambda^c$
into two disjoint sets with $2\leq \# \Lambda \leq n-2$ and we have $S_{n}^{\Lambda}=S_n^{\Lambda^c}$.
Via the natural map
$\pi_{n+1}:  \overline{\mathcal{M}}_{0,n+1} \to  \overline{\mathcal{M}}_{0,n}$
we may view  $\overline{\mathcal{M}}_{0,n+1}$ as the universal curve
and $\pi_{n+1}$ has $n$ sections. The generic point of $S_{n}^{\Lambda}$
corresponds to a stable curve with two rational components, one of which
contains the points marked by $\Lambda$. For pullback by $\pi_{n+1}$ we have the
relation
$$
\pi_{n+1}^{*}(S_n^{\Lambda})= S_{n+1}^{\{\Lambda,n+1\}} \cup S_{n+1}^{\{\Lambda^c,n+1\}} \, .
$$
The $n$ sections of $\pi_{n+1}$
have images $S_{n+1}^{\{i,n+1\}}$ with $i=1,\ldots,n$.

We can collect these boundary divisors on $\overline{\mathcal{M}}_{0,n+1}$ via
$$
T_{n+1,j}= \sum_{\#\Lambda=j} S_{n+1}^{\{\Lambda,n+1\}}, \quad
T_{n+1,j}^c= \sum_{\#\Lambda=j} S_{n+1}^{\{\Lambda^c,n+1\}} \, ,
$$
with the convention that in view of the symmetry we add a factor $1/2$ for even $n$ and $j=n/2$
$$
T_{n+1,n/2}= \frac{1}{2}\sum_{\#\Lambda=n/2} S_{n+1}^{\{\Lambda,n+1\}}, \quad
T_{n+1,n/2}^c= \frac{1}{2} \sum_{\#\Lambda=n/2} S_{n+1}^{\{\Lambda^c,n+1\}} \, .
$$
Later, when a fixed index $k$ is given we will split these divisors as
$T=T(k^{+}) +T(k^{-})$ where $(k^{+})$ (resp.\ $(k^{-})$) indicates that the
sum is taken over $\Lambda$ containing $k$ (resp.\ not containing $k$).
So $T_{n+1,j}(k^{+})= \sum_{\#\Lambda=j, k \in \Lambda} S_{n+1}^{\{\Lambda,n+1\}}$ (and with a factor $1/2$ if $j=n/2$).
\end{section}

\begin{section}{A good model}
We now will work with a `good model' of the universal admissible cover
over $\overline{\mathcal{H}}_{g,2}$. Such a model was constructed in
\cite[Section 4]{vdG-K2}.  We start with the observation
that the space $\overline{\mathcal{H}}_{g,2}$
is not normal, and we therefore normalize it. The result
$\widetilde{\mathcal{H}}_{g,2}$ is now a smooth stack over which we have
a universal curve $\tilde{\mathcal{C}} \to \widetilde{\mathcal{H}}_{g,2}$.

We have a natural map
$h: \widetilde{\mathcal{H}}_{g,2} \to \overline{\mathcal{M}}_{0,b}$
with $b=2g+2$
and the universal curve now fits into a commutative diagram
$$
\begin{xy}
\xymatrix{
&{\tC} \ar[r]^c \ar[d]_{\varpi}& {\bM}_{0,b+1}\ar[d]^{\pi_{b+1}} \\
\overline{\mathcal{H}}_{g,2} & {\tH}_{g,2}\ar[l]_{\nu} \ar[r]^{h} & {\bM}_{0,b} \\
}
\end{xy}
$$
We can construct a proper flat map that extends the relative
canonical morphism
$\C \to {\PP}^1_{{\H}_{g,2}}$
by taking the fibre product ${\PP}$ of ${\bM}_{0,b+1}$ and ${\tH}_{g,2}$
over ${\bM}_{0,b}$ and thus obtain a commutative diagram
$$
\begin{xy}
\xymatrix{
{\tC} \ar[r]^{\alpha} \ar[rd]_{\varpi}& {\PP}\ar[d]^{\varpi'} 
\ar[r]^{c'} &{\bM}_{0,b+1}\ar[d]^{\pi_{b+1}} \\
& {\tH}_{g,2} \ar[r]^{h} & {\bM}_{0,b} \\
}
\end{xy}
$$
The resulting space ${\PP}$ is not smooth, but has rational singularities.
Resolving these in a minimal way gives a model $\widetilde{\PP}$;
taking the resolution $\widetilde{Y}$ of the normalization $Y$ of the
fibre product of ${\tP}$ and $\tC$ over ${\PP}$
gives us finally a commutative diagram
$$
\begin{xy}
\xymatrix{
\widetilde{Y} \ar[r]^f \ar[rd]_t & \widetilde{\PP} \ar[d]^{\pi}\\
& B \ar[r]^{\beta} & \widetilde{\mathcal{H}}_{g,2}\\
}
\end{xy}  \eqno(3)
$$
where $B$ is our base $\widetilde{\mathcal{H}}_{g,2}$ or any other base
mapping to it. We write $\pi$ for the resulting morphism $\widetilde{\PP}
\to B$, $h$ for the natural map $B \to \overline{\mathcal{M}}_{0,b}$
and $\nu$ for $B \to \overline{\mathcal{H}}_{g,2}$.
We refer to \cite[Section 4]{vdG-K2} for additional details.

In the following we will assume that we have a physical family over a
base $B$. We will abuse the notation $\Delta^{\Lambda}$ for the pull
back of the divisor $\Delta^{\Lambda}$
under $\nu: B \to \overline{\mathcal{H}}_{g,2}$.

In the case that $\# \Lambda$ is even, say $\#\Lambda=2i+2$ with $0 \leq i \leq (g-1)/2$,
the pull back of $\Delta^{\Lambda}$ decomposes as
$$
\pi^*(\Delta^{\Lambda})=\Pi^{ \Lambda}+ \Pi^{\Lambda^c}\, ,
$$
corresponding to the two components of a general fibre
of $\pi$, with $\Pi^{\Lambda}$ mapping to $S_{b+1}^{\{ \Lambda,b+1\}}$
under $\widetilde{\PP} \to \overline{\mathcal{M}}_{0,b+1}$,
and similarly $\Pi^{\Lambda^c}$ mapping to $S_{b+1}^{\{ \Lambda^c,b+1\}}$.
Note that we restrict $\#\Lambda$ by $\leq g+1$, hence the notation
$\Pi^{\Lambda}$ should not lead to confusion.

In the case $\# \Lambda$ is odd we find a similar decomposition
$$
\pi^*(\Delta^{\Lambda})=\Pi^{ \Lambda}+ R^{\Lambda} +\Pi^{\Lambda^c}\, ,
$$
corresponding now to the fact that the general fibre of $\pi$ has three components,
one coming from the blowing up.

We notice
$$
h^*(S_b^{\Lambda})= \begin{cases} 
\Delta^{\Lambda} & \# \Lambda \equiv 0\, (\bmod\, 2) \\
2\, \Delta^{\Lambda} & \# \Lambda \equiv 1\,  (\bmod \, 2) \\ \end{cases} 
$$

If we use the notation $\Delta_j= \sum_{\# \Lambda=j} \Delta^{\Lambda}$, 
we find for the tautological classes $\lambda=c_1({\EE})$ and 
$h^*(\psi_k)$, simply denoted
by $\psi_k$ and defined as the first Chern class of the line bundle
given by the cotangent space at the $k$th point 
of our pointed curve ($k=1,\ldots,b$), 
the following formulas  on our base $B$ (see \cite{vdG-K1})
$$
\lambda= \sum_{i=0}^{(g-1)/2} \frac{(i+1)(g-i)}{2(2g+1)} \Delta_{2i+2} + 
\sum_{i=1}^{g/2} \frac{i(g-i)}{2g+1} \Delta_{2i+1}
\eqno(4)
$$
and
$$
\begin{aligned}
\psi_k= \sum_{i=0}^{(g-1)/2} \left( \frac{(g-i)(2g-2i-1)}{g(2g+1)}
\Delta_{2i+2}(k^{+}) +
\frac{(i+1)(2i+1)}{g(2g+1)} \Delta_{2i+2}(k^{-}) \right) & \\
+2 \, \sum_{i=1}^{g/2} \left( \frac{(g-i)(2g-2i+1)}{g(2g+1)} \Delta_{2i+1}(k^{+}) +
 \frac{i(2i+1)}{g(2g+1)} \Delta_{2i+1}(k^{-}) \right) & \\
\end{aligned}
\eqno(5)
$$
where we use the notation $(k^{+})$ (resp.\ $(k^{-})$)
to denote the condition $k\in \Lambda$ (resp.\ $k \not\in \Lambda$)
as above. The relation (4) implies the following.

\begin{corollary}
There exists a scalar-valued modular form
of weight $2(2g+1)$ on the moduli space $\widetilde{\mathcal{H}}_{g,2}$
whose divisor is a union of boundary divisors. 
It descends to the hyperelliptic locus
$\overline{\mathcal{H}}_g$ and corresponds to a power of the discriminant
of the binary form of degree $2g+2$.
\end{corollary}
\end{section}
\begin{section}{Extending the linear system} \label{extending}
The canonical system on a hyperelliptic curve is defined by
the pull back of the sections of the line bundle of
$\mathcal{O}(g-1)$ of degree $g-1$ on the
projective line.
We now try to extend this line bundle over our compactification.

A first attempt would be to consider the divisor
$(g-1)\, \tilde{S}_k$  with $\tilde{S}_k$ the pullback to $\widetilde{\PP}$ of the section
$S_k$ of $\pi_{b+1}: \overline{\mathcal{M}}_{0,b+1}\to 
\overline{\mathcal{M}}_{0,b}$.  Recall the morphism 
$t=\pi f \colon \widetilde{Y} \to B$. We can add a boundary divisor $\Xi_k$ to it such that
$f^*\mathcal{O}_{\widetilde{\PP}}(D_k)$ with $D_k=(g-1)\tilde{S}_k+\Xi_k$
coincides with $\omega_{t}$ on the fibres of $t$, 
namely in view of the intersection numbers take $\Xi_k$ equal to
$$
\begin{aligned}
&\sum_{i=0}^{(g-1)/2} \left(
(g-1-i)\, \Pi_{2i+2}(k^{+}) + i\, \Pi_{2i+2}^c(k^{-})
\right)+ \\
&\sum_{i=1}^{g/2} \left(
(g-i-1)\, \Pi_{2i+1}(k^{+})
- (g-i) \, \Pi^c_{2i+1}(k^{+})
\right)
+  \sum_{i=1}^{g/2} \left(
(i-1)\, \Pi_{2i+1}^c(k^{-}) 
- i\, \Pi_{2i+1}(k^{-})
\right) \\
\end{aligned}
$$
Here $\Pi_j= \sum_{\# \Lambda =j} \Pi^{\Lambda}$ and $\Pi_j^c=\sum_{\# \Lambda=j} \Pi^{\Lambda^c}$
and $(k^{+})$ (resp.\ $(k^{-})$) indicates the condition that $k\in \Lambda$ (resp.\
$k \not\in \Lambda$); moreover, we add a factor $1/2$ in case $j=b/2$.

Now $f^*\mathcal{O}(D_k)$ and $\omega_t$ agree on the fibres of $t$,
so they differ by a pull back under $t=\pi \circ f$, see diagram (2).

To see the above, when e.g.\ $\# \Lambda=2i+2$ is even: 
in that case the fibre of $\tilde{Y}$ over $t$ is as in Fig.\ 1 and 
$\omega_C=(\omega_{C_i}+p+q, \omega_{C_{g-i-1}}+p+q)=(i(p+q), (g-i-1)(p+q))=f^*(i, g-i-1)$, 
where we indicate by $i$ the line bundle of degree $i$ on ${\mathbb P}^1$.
One then checks that with the above choice of 
$\Xi_k$ the restriction of $D_k$ on the corresponding fibre of $\pi$ 
is of type $(i,g-1-i)$. Indeed,
in case $k\in \Lambda$ then $(g-1)\tilde{S}_k+(g-1-i)\Pi_{2i+2}(k^+)$ 
restricts to $((g-1)- (g-i-1), g-i-1)=(i, g-i-1)$ and
in case $k\notin \Lambda$ then $(g-1)\tilde{S}_k+i\Pi^c_{2i+2}(k^-)$ restricts to $(i,(g- 1)-i)=(i,g-i-1)$.
The case where $\# \Lambda =2i+1$ is odd, although a little more complicated, is treated similarly.

We therefore will change $D_k$ by a pull back under $\pi$.
Define  a divisor class on $B$ by
$$
\begin{aligned}
E_k= \frac{2g-1}{2}\psi_k & - \sum_{i=0}^{(g-1)/2} 
\left( 
(g-i-1)\, \Delta_{2i+2}(k^{+}) + i\, \Delta_{2i+2}(k^{-})
\right) \\
& - \sum_{i=1}^{g/2} \left(
(g-i-1)\, \Delta_{2i+1}(k^{+})  +(i-1)\, \Delta_{2i+1}(k^{-})
\right)
\end{aligned} 
$$
and define a line bundle on $\tilde{\PP}$ by
$$
M= \mathcal{O}(D_k+\pi^* E_k) \, . \eqno(6)
$$
\begin{lemma} \label{lemma1}
The line bundle $M$ does not depend on $k$, satisfies
$f^*(M)=\omega_t$  and restricts to the general
fibre ${\PP}^1$ of $\pi$ as $\mathcal{O}(g-1)$.
For $\#\Lambda=2i+2$ its restriction to the
general fibre ${\PP}_1 \cup {\PP}_2$ over $\Delta^{\Lambda}$ is
of degree $(i,g-i-1)$, while for $\# \Lambda=2i+1$ its restriction
to the general fibre ${\PP}_1 \cup R \cup {\PP}_2$  is of degrees
$(i,-1,g-i)$.
\end{lemma}
\begin{proof}
We use the section $\tau_k$ of $t: \widetilde{Y} \to B$ with
$f \tau_k=\tilde{s}_k$ with 
$\tilde{s}_k$ the natural sections of the map $\pi$ with image $\tilde{S}_k$. 
Then we have
$\tau_k^* \omega_t =\psi_k/2$  and
$\tau_k^* f^* D_k=\tilde{s}_k^* D_k$ for which we have
$$
\begin{aligned}
\tilde{s}_k^*D_k= -(g-1)\psi_k + 
\sum_{i=0}^{(g-1)/2} \left(
(g-i-1)\, \Delta_{2i+2}(k^{+}) + i\, \Delta_{2i+2}(k^{-})
\right) & \\
+ \sum_{i=1}^{g/2} \left(
(g-1-i)\, \Delta_{2i+1}(k^{+}) +(i-1)\, \Delta_{2i+1}(k^{-})
\right)& \, . \\
\end{aligned}
$$
From this we obtain $\tau_k^*(\omega_t)-\tau_k^*f^*D_k=\pi^*(E_k)$,
so that $\omega_t= f^*(\mathcal{O}(D_k+E_k))$. We also see that the
restriction of $M$ on the fibres of $\pi$ does not depend on $k$.
Moreover, we have
$$
\tilde{s}_j^*(D_k+\pi^*E_k)=\tau_j^* f^* (D_k+E_k)= \tau_j^*(\omega_t) =\psi_j/2
=\tilde{s}_j^*(D_j+\pi^*E_j)\, ,
$$
showing that the restrictions of $\mathcal{O}(D_k+\pi^*E_k)$
and $\mathcal{O}(D_j+\pi^* E_j)$
agree on $\tilde{S}_j$. The restrictions of the fibres of $\pi$ over
the general points of $\Delta_b^{\Lambda}$ are easily checked.
\end{proof}

We now want to compare $\pi_*(M)$ with the Hodge bundle
${\EE}=t_* (\omega_t)$ on $B$. The next proposition shows that
these agree up to  codimension $2$.

\begin{proposition}\label{codim2}
We have an exact sequence
$ 0 \to \pi_*(M)\to {\EE} \to \mathcal{T} \to 0$,
where $\mathcal{T}$ is a coherent sheaf that is a torsion sheaf supported on the boundary.
Moreover, we have
$c_1(\pi_*(M))=\lambda$.
\end{proposition}
\begin{proof} By Lemma \ref{lemma1} we have
$\omega_t=f^*(M)$. But $R^1\pi_*(M)=(0)$, so we have
$$
\pi_*(M\otimes f_*\mathcal{O}_{\tilde{Y}})=
\pi_*f_*(f^*(M))=\pi_*f_*(\omega_t)=t_*(\omega_t)
\, .
$$
We have an exact sequence
$
0\to \mathcal{O}_{\tilde{\PP}}\to f_* \mathcal{O}_{\tilde{Y}} \to 
\mathcal{F} \to 0
$
with $\mathcal{F}$ a coherent sheaf of rank~$1$ that restricted to the
smooth fibers of $\pi$ has degree $-(g+1)$,
as one sees by
applying Riemann-Roch to $f$ and $\mathcal{O}_{\tilde{Y}}$.
Tensoring the sequence
with $M$ and applying $\pi_*$ gives the exact sequence
$$
0 \to \pi_*(M) \to \pi_*(M \otimes f_*\mathcal{O}_{\tilde{Y}}) \to 
\pi_*(M\otimes \mathcal{F}) \to 0\, .
$$
On the smooth fibers of $\pi$ the sheaf
$M \otimes \mathcal{F}$ restricts to a
line bundle of degree $(g-1)-(g+1)=-2$,
hence $\pi_*(M\otimes \mathcal{F})$ is a
torsion sheaf.

We now calculate $c_1(\pi_*(M))$.
We apply Grothendieck-Riemann-Roch to $\pi$ and $\mathcal{O}(D_k)$. It says
$$
{\rm ch}(\pi_!(\mathcal{O}(D_k)))=\pi_*({\rm ch}(\mathcal{O}(D_k))){\rm Td}^{\vee}(\omega_{\pi})\, ,
$$
which by  $\pi_*({\rm Td}^{\vee}_2(\omega_{\pi}))=0$ gives
$$
c_1(\pi_*(\mathcal{O}(D_k)))=\frac{1}{2} 
\pi_{*}(-D_k\, \omega_{\pi} +  D_k^2)\, . 
$$
We calculate
$$
\pi_*(D_k \omega_{\pi})= (g-1)\psi_k - \sum_{i=0}^{(g-1)/2} 
\left(
(g-i-1)\, \Delta_{2i+2}(k^{+})+i \, \Delta_{2i+2}(k^{-})
\right) 
+\sum_{i=1}^{g/2} \Delta_{2i+1} \, , \eqno(7)
$$
and
$$
\begin{aligned}
\pi_*(D_k^2)=&
-(g-1)^2\psi_k + \sum_{i=0}^{(g-1)/2} 
(g-i-1)(g+i-1)\, \Delta_{2i+2}(k^{+}) \\
& +\sum_{i=0}^{(g-1)/2} (2g-2-i)i \, \Delta_{2i+2}(k^{-}) 
+ \sum_{i=1}^{g/2} ((2g-i-1)(i-1)-i^2)  
\, \Delta_{2i+1}\,  . \\ \
\end{aligned} \eqno(8)
$$
Adding $\pi_*(\pi^* E_k)$ gives
$$
\begin{aligned}
c_1(\pi_*(M))= \frac{g^2}{2} \, \psi_k -\frac{1}{2} \sum_{i=0}^{(g-1)/2}
\left( 
(g-i-1)(g-i) \, \Delta_{2i+2}(k^{+}) +
i(i+1)\, \Delta_{2i+2}(k^{-}) 
\right)&\\
- \sum_{i=1}^{g/2} \left(
(g-i)^2\, \Delta_{2i+1}(k^{+}) + i^2 \, \Delta_{2i+1}(k^{-})
\right)\, . &\\
\end{aligned}
$$
Substituting the formula for $\psi_k$ we find
$$
c_1(\pi_*(M))= \sum_{i=0}^{(g-1)/2} \frac{(g-i)(i+1)}{2(2g+1)} 
\Delta_{2i+2} +
\sum_{i=1}^{g/2} \frac{i(g-i)}{2g+1} \Delta_{2i+1} =\lambda\, .
$$
\end{proof}

The line bundle $M$ on $\tilde{\PP}$ is not base point free as Proposition
\ref{codim2} shows; the restriction to the $R$-part has negative degree.
We can make it base point free by defining
$$
N=M(-R)=\mathcal{O}(D_k+\pi^* E_k -R) \, . \eqno(9)
$$
Now the restriction of $N$ to a general fibre over $\Delta_{2i+1}$,
which is a chain of three rational curves ${\PP}_1,R,{\PP}_2$,
has degrees $(i-1,1,g-i-1)$ and one checks that $N$ is base point free.

\begin{lemma} \label{pistarN=E}
Up to codimension~$2$ we have on $B$ that $\pi_*(N)={\EE}$.
\end{lemma}
\begin{proof}
We have $R^1\pi_*(N)=0$. Therefore the exact sequence
$0 \to N \to M \to M_{|R} \to 0$ yields the exact sequence
$$
0 \to \pi_*(N) \to \pi_*(M) \to \pi_*(M_{|R}) \to 0 \, .
$$
We now show that $c_1(\pi_*(N))=c_1(\pi_*(M))$. Since $R^1\pi_*(M)=0=
R^1\pi_*(N)$ we find by Grothendieck-Riemann-Roch that
$$
c_1(\pi_*(M))=\frac{1}{2} \pi_*(c_1(M)^2-c_1(M) \omega_{\pi}), \quad
c_1(\pi_*(N))=\frac{1}{2} \pi_*(c_1(N)^2-c_1(N) \omega_{\pi})\, .
$$
By the definition of $N$ and the fact that $R$ is a $(-2)$-curve if
we take a base $B$ of dimension~$1$, and thus has intersection number
$0$ with a fibre, we have
$$
\pi_*(c_1(N))^2= \pi_*(c_1(M))^2
$$
and $c_1(N) \,  \omega_{\pi}=c_1(M) \, \omega_{\pi}$
since the restriction of $\omega_{\pi}$ to $R$ is trivial.
\end{proof}
\end{section}
\begin{section}{The rational normal curve}\label{rationalnormalcurve}
The image of a hyperelliptic curve by the canonical map is a
rational normal curve, that is, ${\PP}^1$ embedded in ${\PP}^{g-1}$
via the linear system of degree $g-1$.

In our setting we can see the rational normal curve and
its degenerations
using the extension $N$ of the line bundle of degree $g-1$, as defined in (9), to
the compactification as constructed in the preceding section.

We let $u: {\PP}({\EE}) \to B$ be the natural projection. Now $N$ is
base point free and up to codimension $2$ we have $\pi_*(N)={\EE}$,
so the global-to-local map $\pi^*\pi_{*}(N) \to N$ induces a surjective
map $\nu: \pi^*({\EE}) \to N$ over $\tilde{\PP}$. This induces a morphism
$\phi: \tilde{\PP} \to {\PP}({\EE})$ by associating to a point of
$\tilde{\PP}$ the kernel of $\nu$. It fits into a diagram
$$
\begin{xy}
\xymatrix{
\tilde{\PP} \ar[r]^{\phi} \ar[dr]_{\pi}  & {\PP}({\EE}) \ar[d]^u \\
& B \\
}
\end{xy}
$$
\begin{proposition} \label{geometric-meaning}
For a point of $B$ with smooth fibre under $\pi$ the image
of $\phi$ is a rational normal curve of degree $g-1$. For a general
point $\beta \in \Delta_{2i+2}$ with fibre ${\PP}_1 \cup {\PP}_2$
the image is a union of two rational normal curves of
degree $i$ and $g-i-1$.
For a general point $\beta \in \Delta_{2i+1}$ with fibre
${\PP}_1,R,{\PP}_2$ the image is a union of three rational normal
curves of degree $i-1$, $1$ and $g-i-1$. Here we interpret the case
of degree $0$ as a contracted curve.
\end{proposition}
\begin{proof} The proposition follows almost immediately from Lemma \ref{lemma1}.
\end{proof}

\begin{remark}\label{caseg=2}
If $i=1$ then ${\PP}_1$ is contracted. If also $g=2$ then both ${\PP}_1$
and ${\PP}_2$ are contracted and the image of $R$ coincides with the fibre of ${\PP}({\EE})$.
\end{remark}
\begin{remark}

The sections $\tilde{s}_i: B \to \tilde{\PP}$ for $i=1,\ldots,b$ induce sections
$\sigma_i=\phi \circ s_i: B \to {\PP}({\EE})$ by
sending $\beta$ to the kernel of ${\EE}=\tilde{s}_i^* \pi^* ({\EE}) \to
s_i^*(N)$.
\end{remark}

\begin{remark}\label{remark-blowdown}
In the case $g=2$ the map $\phi$ is a birational map
$\tilde{\PP}\to {\PP}({\EE})$ that blows down boundary components.
More precisely, over $\Delta_{2}$ it blows down ${\Pi}_{2}$
and over $\Delta_{3}$ the components supported at ${\Pi}_{3}=
{\Pi}^c_{3}$.
\end{remark}
\end{section}
\begin{section}{Symmetrization}\label{symmetrization}
We have been working with the moduli space ${\mathcal{H}}_{g,2}$
and $\mathcal{M}_{0,b}$ and their compactifications. Here
the symmetric group $\mathfrak{S}_{b}$ acts. We therefore make our construction symmetric.

We put $D=\sum_{k=1}^b D_k$ and $E=\sum_{k=1}^b E_k$ and set
$$
\tilde{M}=\mathcal{O}(D+E)\,, \quad \psi=\sum_{k=1}^b \psi_k\, ,
\quad \text{\rm and} \quad 
\tilde{S}=\sum_{k=1}^b \tilde{S}_k\, .
$$
We find
$$
\psi= 4 \, \sum_{i=0}^{(g-1)/2} \frac{(g-i)(i+1)}{2g+1} \, \Delta_{2i+2}
+ 2\, \sum_{i=1}^{g/2} \frac{(2g-2i+1)(2i+1)}{2g+1} \, \Delta_{2i+1}
$$
and
$$
\begin{aligned}
D= (g-1)\tilde{S} +2 \sum_{i=0}^{(g-1)/2} 
\left( 
(g-i-1)(i+1) \Pi_{2i+2}+ 
i\, (g-i) \, \Pi^c_{2i+2}\right) +\\
\sum_{i=1}^{g/2} \left( (g-4i-1)\Pi_{2i+1}-(3g-4i+1)\, \Pi^c_{2i+1}\right) & \\
\end{aligned}
$$
and
$$
E=  \frac{2g-1}{2} \psi 
 -2 \sum_{i=0}^{(g-1)/2} 
\left( (g-i)(2i+1)-(i+1)\right) \Delta_{2i+2} 
 -\sum_{i=1}^{g/2} \left(4i\, (g-i)-(g+2)\right) \Delta_{2i+1} 
$$
\end{section}
\begin{section}{The Case of Hyperelliptic Genus Three}
The Hurwitz space $\mathcal{H}_{3,2}$ admits a compactification $\overline{\mathcal{H}}_{3,2}$
with boundary components $\Delta^{\Lambda}$ with $\# \Lambda \in \{ 2,3,4\}$. Taking the components
with $\Lambda$ of fixed cardinality together gives boundary components $\Delta_2, \Delta_3$ and $\Delta_4$. Under the morphism
$\overline{\mathcal{H}}_{3,2}\to \overline{\mathcal{M}}_3$ the components $\Delta_2$ and $\Delta_4$ are mapped to
$\delta_0$, while $\Delta_3$ goes to $\delta_1$. The formulas (4) and (5) specialize to
$$
\lambda= \frac{3}{14} \Delta_2+\frac{2}{7} \Delta_3 + \frac{2}{7}\Delta_4   \eqno(10)
$$
and
$$
\psi_k= \frac{5}{7}\,  \Delta_2(k^{+})+\frac{1}{21}\, \Delta_2(k^{-}) +\frac{20}{21}\, \Delta_3(k^{+})+
\frac{2}{7}\,  \Delta_3(k^{-}) +\frac{2}{7}\, \Delta_4 \, .
$$
\begin{remark}\label{weight14}
The equation (10) shows that on $\overline{\mathcal{H}}_{3,2}$ there exists a scalar-valued modular
form of weight $14$ whose square equals $\chi_{28}$, a form mentioned in Section \ref{hyperellg3}. 
Since on $\overline{\mathcal{H}}_3$ we have
$28\, \lambda= 3\, \delta_0+8\, \delta_1+8\, \zeta_1$, an integral class not divisible by $2$, 
there is not a modular form of weight $14$
on $\overline{\mathcal{H}}_3$ with square $\chi_{28}$. Compare with Cornalba's formula
at the end of Section \ref{hyperell-stack}.
\end{remark}

We have the line bundle $M$ on $\tilde{\PP}$ 
defined in (6) corresponding to the divisor class
$D_k+E_k$ given by
$$
D_k= 2\, \tilde{S}_k+2\, \Pi_2(k^{+}) + 2\, \Pi_4(k^{+}) + \Pi_3(k^{+})-2\, \Pi_3^c(k^{+})-\Pi_3(k^{-})
$$
and
$$
E_k=\frac{5}{2}\psi_k- (2\, \Delta_2(k^{+})+\Delta_3(k^{+})+\Delta_4)\, ,
$$
where $\psi_k$ is given in (6). Define the rational divisor class
$$
U:= \frac{1}{14} \Delta_2 + \frac{3}{7}(\Delta_3+\Delta_4) = \frac{3}{2} \psi_k -(\Delta_2(k^{+})+\Delta_3(k^{+})) \, . 
$$
The divisor class of $D_k+E_k$ is independent of $k$ as observed in Lemma \ref{lemma1},
but this can be seen also directly from the next lemma.
\begin{lemma} \label{Nformula}
We have the linear equivalence $D_k+E_k \sim -\omega_{\pi}+\Pi_2+\Pi_3+ \pi^*(U)$.
\end{lemma}
\begin{proof}
One checks that $-\omega_{\pi}+\Pi_2+\Pi_3$ and $D_k+E_k$ have the same restriction to fibres of $\pi$.
We have $s_k^*(-\omega_{\pi}+\Pi_2+\Pi_3)=-\psi_k+\Delta_2(k^{+})+\Delta_3(k^{+})$
and $s_k^*(D_k+E_k)=\psi_k/2$.
\end{proof}

Let $Q$ be the image of $\phi: \tilde{\PP} \to {\PP}({\EE})$, see Proposition 
\ref{geometric-meaning}. The map $\phi $ is the composition of a map 
$\phi': \tilde{\PP}\to Q$ with the inclusion map $\iota: Q \hookrightarrow {\PP}(\EE)$. 
The generic fibre of $Q\to B$ is a conic, hence
$\mathcal{O}(Q)=\mathcal{O}(2)\otimes \mathcal{O}(u^*A)$
for some divisor $A$ on $B$. We determine $A$.
\begin{lemma}\label{Qlemma}
On ${\PP}({\EE})$ we have the linear equivalence 
$$
[Q] \sim [\mathcal{O}(2)]+u^*(4\, \lambda -(\Delta_2+\Delta_3+\Delta_4))\, .
$$
\end{lemma}
\begin{proof}
We have $\omega_{\PP({\EE})}\otimes u^*(\omega_B^{-1})=\mathcal{O}(-3)\otimes u^*(\det{\EE})$
and by adjunction $\omega_Q=\iota^*(\mathcal{O}(Q) \otimes \omega_{\PP({\EE})})$.
Since $\phi'$ is a blow down we have $\omega_{\tilde{\PP}}=(\phi')^*\omega_Q
\otimes \mathcal{O}(\Pi_2+\Pi_3)$.
We get
$$
\begin{aligned}
\phi^*(\mathcal{O}(Q))=& {\phi'}^*\omega_Q \otimes \phi^*\omega^{-1}_{\PP({\EE})}\\
=& \omega_{\tilde{\PP}}\otimes \mathcal{O}(-\Pi_2-\Pi_3) \otimes \phi^*\mathcal{O}(3) 
\otimes \pi^*\det({\EE})^{-1} \otimes \pi^* \omega_B^{-1} \\
=&\omega_{\pi} \otimes \mathcal{O}(-\Pi_2-\Pi_3) \otimes 
\phi^*\mathcal{O}(3) \otimes \pi^* \det({\EE})^{-1}\, . \\
\end{aligned}
$$
On the other hand we have
$
\phi^*\mathcal{O}(Q) =\phi^*\mathcal{O}(2) \otimes \mathcal{O}(\pi^*A)
$
and $\phi^*\mathcal{O}(1)=N$, hence we get
$$
\mathcal{O}(u^*A)=
N \otimes \omega_{\pi} \otimes \mathcal{O}(-\Pi_2-\Pi_3) \otimes \pi^* \det({\EE})^{-1} \, . \eqno(11)
$$
By Lemma \ref{Nformula} we have $N = \omega_{\pi}^{-1} \otimes \mathcal{O}(\Pi_2+\Pi_3)
\otimes {O}(U)$.
Substituting this in (11) we get the desired result.
\end{proof}
The effective divisor $Q$ yields a modular form and Lemma \ref{Qlemma}
gives its weight.
\begin{corollary}\label{chi204}
The effective divisor $Q$ on ${\PP}({\EE})$ 
defines a modular form $\chi_{2,0,4}$ on $\tilde{\mathcal{H}}_{3,2}$ of weight $(2,0,4)$, 
that is,
a non-zero section of ${\rm Sym}^2({\EE}) \otimes \det({\EE})^4$.
\end{corollary}
Since the divisor $\Delta_2+\Delta_3+\Delta_4$ is not a pull back
from the moduli space $\overline{\mathcal{H}}_3$,
the modular form does not descend to $\overline{\mathcal{H}}_3$.
Recall that the modular form $\chi_{4,0,8}$ restricted to the hyperelliptic locus 
was associated to a divisor $D$ that equals $2\, Q$.

\smallskip

\begin{remark} \label{remarkg=2}
In the same vein as above we can determine in
an alternative way the result of Proposition \ref{classofA}
on class of the closure $\overline{D}$ of 
the ramification divisor $D$ of the universal genus $2$ curve.
By the theory of admissible covers there is a natural map 
$\overline{\mathcal H}_{2,2}\to \overline{\mathcal{M}}_2$ with the 
property that the pull back of the Hodge bundle on $\overline{\mathcal{M}}_2$ 
is the Hodge bundle on $\overline{\mathcal H}_{2,2}$ associated 
to the corresponding family of admissible covers. Hence the pull back 
of the $\mathcal{O}(1)$ of the bundle ${\mathbb P}({\mathbb E}) \to \overline{\mathcal{M}}_2$ 
equals the $\mathcal{O}(1)$ of the bundle 
${\PP}({\EE}) \to \overline{\mathcal H}_{2,2}$. 
Let $\Sigma=\phi_*(\sum_{k=1}^6\tilde{S}_k) $, with $\phi: {\PP}\to {\PP}({\EE})$
the map defined in Section  \ref{rationalnormalcurve}.
By geometry, the pull back of $\overline{D}$ to the bundle ${\PP}({\EE})$
over $\overline{\mathcal H}_{2,2}$ equals $\Sigma$. By Remark  \ref{remark-blowdown}
we have $\phi^*\Sigma = \tilde{S}+2\, \Pi_2+6\, \Pi_3$. By using the formulas of Section  \ref{symmetrization}, we have for $g=2$:
$$
\phi^*\mathcal{O}(6)= \tilde{M}-6\, R=\tilde{S}+\frac{12}{5}\Pi_2+\frac{2}{5}\Pi_2^c+\frac{24}{5}\Pi_3-\frac{3}{5} R\, .
$$
We now write $[\overline{D}]=\mathcal{O}(6)+u^*(a\, \delta_0+b\, \delta_1)$.
By pulling back to $\tilde{\PP}$ and using the above formulas, we get (we refer to the diagram in
Section \ref{rationalnormalcurve} for notation)
$$
\tilde{S}+\frac{12}{5}\Pi_2+\frac{2}{5}\Pi_2^c+\frac{24}{5}\Pi_3-\frac{3}{5}R+
\pi^*(2 a\, \Delta_0+b\, \Delta_3)
= \tilde{S}+2\, \Pi_2+6\, \Pi_3 \, .
$$
This implies
$$
\pi^*(2a\, \Delta_0+b\, \Delta_3)=-\frac{2}{5}(\Pi_2+\Pi_2^c)+\frac{3}{5}(2\, \Pi_3+R)
=\pi^*(-\frac{2}{5}\Delta_2+\frac{3}{5}\Delta_3 )\, ,
$$
hence $a=-1/5$ and $b=3/5$ and the result follows by using the formula
$10\, \lambda= \delta_0+2\,\delta_1$.
\end{remark}
\end{section}
\begin{section}{Comparison with the Hodge Bundle}

We know by Lemma \ref{pistarN=E} that the line bundle $N=\mathcal{O}_{\widetilde{\PP}}(D_k+E_k-R)$ 
on $\tilde{\PP}$ over $\widetilde{\mathcal{H}}_{3,2}$ has the property that
$\pi_*(N)\cong {\EE}$ up to codimension $2$. 
We now deal with the push forward of the tensor powers of $N$.
\begin{lemma}
We have for $m \in {\ZZ}_{\geq 1}$
$$
c_1(\pi_*(N^{\otimes m}))=\frac{2\, m^2+m}{14} \, \Delta_2  + \frac{5\, m^2-m}{14} (\Delta_3+\Delta_4)\, .
$$
\end{lemma}
\begin{proof}
We apply Grothendieck-Riemann-Roch to $\pi$ and $N^{\otimes m}$
as in  (10) in the proof of Proposition \ref{codim2}.
Recall that $N$ corresponds to the divisor(class) $D_k+E_k-R$.
We use that $R^1\pi_* N^{\otimes m}=0$ for all $m$ and find
$$
c_1(\pi_*(N^{\otimes m}))=\frac{1}{2} \pi_*\left( m^2 (D_k+E_k-R)^2 -m\,  \omega_{\pi} \cdot (D_k+E_k-R) \right)
$$
and using the relations (8) and (9) of the proof of Proposition \ref{codim2} we get
$$
c_1(\pi_*(N^{\otimes m}))=\frac{2\, m^2+m}{14} \, \Delta_2  + \frac{5\, m^2-m}{14} (\Delta_3+\Delta_4)
$$
as required.
\end{proof}
\begin{proposition}\label{symm-2}
On $B$ we have the exact sequence
$$
0\to {\rm Sym}^{m-2}({\EE})\otimes \mathcal{O}(-A) \to {\rm Sym}^{m}({\EE}) \to \pi_*(N^{\otimes m}) \to 0 \, ,
$$
with $A=4\, \lambda -(\Delta_2+\Delta_3+\Delta_4)$.
\end{proposition}
\begin{proof}
By Lemma \ref{Qlemma} we have on ${\PP}({\EE})$ the exact sequence
$$
0 \to \mathcal{O}(m-2)\otimes u^*\mathcal{O}(-A) \to \mathcal{O}(m) \to \mathcal{O}(m)_{|Q}
\to 0
$$
Applying $u_*$ and observing that $R^1u_*\mathcal{O}(m-2)$ vanishes gives the result.
\end{proof}
A section of
${\rm Sym}^j({\EE})\otimes \det({\EE})^k$ over $\mathcal{H}_3$
pulls back to the stack $[W^0_{8,-2}/({\rm GL}(W)/(\pm 1_W))]$ 
as a section of ${\rm Sym}^j({\rm Sym}^2(W))\otimes \det(W)^{k/2}$ for even $k$. We have an
isotypical decomposition
$$
{\rm Sym}^j({\rm Sym}^2(W))=\oplus_{n=0}^{\lfloor j/2 \rfloor} {\rm Sym}^{2j-4n}(W)\otimes \det(W)^{2n} \, ,
$$
where we assume here and in the rest of this section 
that the characteristic is $0$ or not $2$ and 
high enough for this identity to hold (or use divided powers as in \cite[3.1]{AFPRW}).
A section of
${\rm Sym}^j({\EE})\otimes \det({\EE})^k$ over $\mathcal{M}_3^{nh}$ pulls back to 
$[V_{4,0,-1}/{\rm GL}(V)]$, where we now write $V$ for the standard space of dimension $3$.
An identification $V\cong {\rm Sym}^2(W)$ corresponds to an embedding ${\PP}^1 \hookrightarrow {\PP}^2$
with image a smooth quadric. If we view $V$ with basis $x,y,z$, the kernel of the projection 
$$
{\rm Sym}^j(V)={\rm Sym}^j({\rm Sym}^2(W)) \to {\rm Sym}^{2j}(W)
$$
consists of the polynomials of degree $j$ in $x,y,z$ that vanish on the quadric.
Thus in view of the isotypical decomposition above
the exact sequence
$$
0\to {\rm Sym}^{m-2}({\EE})\otimes \mathcal{O}(-A) \to {\rm Sym}^{m}({\EE}) \to \pi_*(N^{\otimes m}) \to 0 \, . 
$$
corresponds to (the pullback to $W_{8,-2}^0$ of) an exact sequence 
$$
0 \to \left({\rm Sym}^{m-2}({\rm Sym}^2W)\right) \otimes \det(W)^2 \to {\rm Sym}^{m}({\rm Sym}^2W) \to {\rm Sym}^{2m}(W) \to 0 \, . 
$$
The section $\chi_{4,0,8}$ of ${\rm Sym}^4({\EE})\otimes \det({\EE})^8$
restricted to the hyperelliptic locus  allows three projections according to the decomposition
$$
{\rm Sym}^4({\rm Sym}^2 W) \otimes \det(W)^{24} = W_{8,24} \oplus W_{4,26} \oplus W_{0,28}\, .
\eqno(12)
$$
\begin{lemma}\label{3summands} 
The projections to the three summands in (12) of the pull back of $\chi_{4,0,8}$ 
to $\mathcal{H}_{3,2}$ define modular forms on $\overline{\mathcal{H}}_{3,2}$ of weights
$(4,0,8)$, $(2,0,4)$ and $(0,0,14)$ and these are up to a scalar given by 
the covariants $f_{8,-2}\, \mathfrak{d}$, $f_{4,-1} \, \mathfrak{d}$ 
and the discriminant $\mathfrak{d}$.
\end{lemma}
\begin{proof}
The identification of ${\EE}$ with ${\rm Sym}^2(W)$ corresponds to the embedding of ${\PP}^1$
as a conic $C$ in ${\PP}^2$. A ternary quartic $Q$ contains $C$ either $0$, $1$ or $2$ times,
say $Q=mC+R$ with $0\leq m \leq 2$. 
The three projections correspond to $R \cap C$ and give the
universal binary octic, the universal binary quartic and $1$ up to twisting. The first projection was
identified in Proposition \ref{binary_octic}. The argument for the second is similar, while the third
descends to $\overline{\mathcal{H}}_3$ and does not vanish on $\mathcal{H}_3$. Therefore it must be
a multiple of the disciminant. Taking into account
the action of ${\rm GL}_2/{\pm 1_W}$ we get the indicated weights (namely $2(14+\epsilon)$ with $\epsilon=-2,-1,0$).
\end{proof}
\end{section}
\begin{section}{More Modular Forms for Genus Three}\label{moregenus3}
We will use more effective divisors on projectivized Hodge bundles 
to produce more modular forms. 
Note that the connection between divisors on projectivized Hodge bundles and modular forms can also
be used in the other direction: obtaining results on cycle classes using modular forms.
We give a few examples.
To a canonical quartic plane curve $C$ we can associate a curve $\check{S}$  in the dual plane of lines intersecting
$C$ equianharmonically. It corresponds to a contravariant (concomitant)  $\sigma$ of the ternary quartic 
given 
by Salmon in \cite[p.\ 264]{Salmon}  and it is defined by
an equivariant ${\rm GL}(3)$ embedding 
$W[4,4,0] \hookrightarrow {\rm Sym}^2({\rm Sym}^4(W))$. It gives rise to a divisor in
${\PP}({\EE}^{\vee})$  and a modular form $\chi_{0,4,16}$ of weight $(0,4,16)$. 
We refer to \cite[p. 54]{CFG2} for the
relation between invariant theory of ternary quartics and modular forms. 
The Siegel modular form $\chi_{0,4,16}$ vanishes with order $2$ at infinity 
and order $4$ along the locus $\mathcal{A}_{2,1}$ of decomposable abelian threefolds. 
With $\check{u}: {\PP}({\EE}^{\vee})\to \overline{\mathcal{M}}_3$ the projection we have 
$\check{u}_*(\mathcal{O}_{{\PP}({\EE}^{\vee})}(1))={\EE}^{\vee}\cong
\wedge^2{\EE} \otimes \det({\EE})^{-1}$
and we 
thus find an effective divisor on ${\PP}({\EE}^{\vee})$ over $\tilde{\mathcal{A}}_3$ with class
$[\check{S}]= [\mathcal{O}_{{\PP}({\EE}^{\vee})}(4)]+20\, \lambda -2\, \delta$ and it vanishes
with multiplicity $4$ along $\mathcal{A}_{2,1}$. We thus find on ${\PP}({\EE}^{\vee})$ over $\overline{\mathcal{M}}_3$
a relation 
$$
[\check{S}]= [\mathcal{O}_{{\PP}({\EE}^{\vee})}(4)]+20\, \lambda -2\, \delta_0 - 4 \, \delta_1\, ,
$$
where we identify $\lambda$ and $\delta_i$ with their pullbacks to ${\PP}({\EE}^{\vee})$.
Similarly, in the dual plane we have the sextic $\check{T}$ of lines intersecting the quartic curve 
in a quadruple of points with 
$j$-invariant $1728$. The corresponding concomitant $\tau$ corresponds to 
$W[6,6,0] \hookrightarrow {\rm Sym}^3({\rm Sym}^4(W))$ and defines a
modular form of weight $(0,6,24)$ vanishing with multiplicity
$3$ at infinity and multiplicity $6$ along $\mathcal{A}_{2,1}$. We thus get a cycle relation
$$
[\check{T}]= [\mathcal{O}_{{\PP}({\EE}^{\vee})}(6)]+30\, \lambda -3\, \delta_0 - 6 \, \delta_1 \, .
$$
The concomitant $\sigma^3-27\, \tau^2$ vanishes on the locus of double conics and the 
corresponding modular form of weight $(0,12,48)$ is divisible by $\chi_{18}^2$ as can be checked using
the methods of \cite{CFG2}. 
Dividing by $\chi_{18}^2$ 
gives a cusp form of weight $(0,12,12)$ vanishing with multiplicity $2$ at infinity and
multiplicity $3$ along $\mathcal{A}_{2,1}$. It is classically known (see e.g. \cite[p.\ 43]{Clebsch}) that this concomitant 
defines the dual curve $\check{C}$ to the canonical image $C$ in ${\PP}({\EE})$.
We thus find an effective divisor in ${\PP}({\EE}^{\vee})$ containing the closure of the
dual curve with class
$$
12 \, [\mathcal{O}_{{\PP}({\EE}^{\vee})}(1)] +24\, \lambda - 2\, \delta_0 -3 \, \delta_1 \, .
$$
This effective divisor class can also be given by the cycle
$$
B=\{ (C,\eta)\in {\PP}({\EE}^{\vee}): \, \text{${\rm div}(\eta)$ has a point of multiplicity $2$} \}
$$
over $\mathcal{M}_3$ and
Korotkin and Zograf in \cite[Thm. 1]{KZ} determined the class of its closure $\overline{B}$
$$
[\overline{B}]= 12 \, [\mathcal{O}_{{\PP}({\EE}^{\vee})}(1)] +24\, \lambda - 2\, \delta_0 -3 \, \delta_1 \, .
$$
Another example of an effective divisor for genus $3$ 
is provided by the Weierstrass divisor $W$ with class
$$
[\overline{W}]= 24\, [\mathcal{O}_{{\PP}({\EE}^{\vee})}(1)]+68\, \lambda -6\, \delta_0 -12 \, \delta_1 
$$
as given by Gheorghita in \cite{Gheorghita}.
Here we get a section of
$$
{\rm Sym}^{24}(\wedge^2 {\EE}) \otimes \det({\EE})^{44}(-6\, \delta_0 -12 \, \delta_1)
$$
This gives  a Teichm\"uller modular form of weight $(0,24,44)$ vanishing with multiplicity $6$  at the cusp.
It descends to a Siegel modular form.
\begin{corollary} The dual of the canonical curve
defines a Siegel modular cusp form of degree $3$ of weight $(0,12,12)$ vanishing with multiplicity
$2$ at infinity. The Weierstrass divisor defines a cusp form of weight $(0,24,44)$ vanishing with multiplicity $6$ at infinity.
\end{corollary}
\end{section}
\begin{section}{The hypertangent divisor}\label{hypertangentsection}
A generic canonically embedded
curve $C$ of genus $3$ has $24$ (Weierstrass) points where the tangent line
intersects $C$ with multiplicity $3$. The union of these $24$ lines forms a divisor
in ${\PP}^2$. Taking the closure of this divisor for the universal curve over $\mathcal{M}_3$
defines a divisor $H$ in ${\PP}({\EE})$ over $\overline{\mathcal{M}}_3$
which we call the hypertangent line divisor. 
We calculate the class of this divisor over $\overline{\mathcal{M}}_3$
and also calculate the class of a corresponding divisor over $\overline{\mathcal{H}}_{3,2}$. 

The calculation over $\overline{\mathcal{M}}_3$ 
uses the divisors $\check{S}$ and $\check{T}$ in ${\PP}({\EE}^{\vee})$ over $\overline{\mathcal{M}}_3$
as defined in the preceding section.
It is a classical result that the intersection $\check{S} \cdot \check{T}$ in the generic fibre is the $0$-cycle consisting of the $24$ 
points defining the $24$ hyperflexes of the generic curve $C$, see \cite{Salmon}. We consider the incidence variety
$$
I=\{ (p,\ell)\in {\PP}({\EE})\times_{\overline{\mathcal{M}}_3} {\PP}({\EE}^{\vee}): p\in \ell \}
$$
Let $\rho: I \to {\PP}({\EE})$ and $\check{\rho}: I\to {\PP}({\EE}^{\vee})$
be the two projections fitting in the commutative diagram
$$
\begin{xy}
\xymatrix{
I \ar[d]^{\rho} \ar[r]^{\check{\rho}} & {\PP}({\EE}^{\vee})\ar[d]^{\check{u}} \\
{\PP}({\EE}) \ar[r]^u & \overline{\mathcal{M}}_3 \\
}
\end{xy}
$$
We have the tautological sequence on ${\PP}({\EE})$
$$
0 \to F \to u^*({\EE}) \to \mathcal{O}_{{\PP}({\EE})}(1) \to 0
$$
and a similar one on ${\PP}({\EE}^{\vee})$
$$
0 \to \check{F} \to \check{u}^*({\EE}^{\vee}) \to \mathcal{O}_{{\PP}({\EE}^{\vee})}(1) \to 0\, .
$$
Now note that $I$ can be identified with the ${\PP}^1$-bundle ${\PP}({{F}^{\vee}})$ 
on ${\PP}({\EE})$,
but also with the ${\PP}^1$-bundle ${\PP}(\check{F}^{\vee})$ on ${\PP}({\EE}^{\vee})$.

The tautological inclusion $F \to u^*{\EE}$ induces a surjection $u^*{\EE}^{\vee} \to F^{\vee}$ and this gives
an inclusion ${\PP}(F^{\vee}) \to {\PP}(u^*{\EE}^{\vee})$ of projective bundles over ${\PP}({\EE})$
which composed with natural map ${\PP}(u^*{\EE}^{\vee}) \to {\PP}({\EE}^{\vee})$ gives the map
$\check{\rho}: I={\PP}(F^{\vee}) \to {\PP}({\EE}^{\vee})$.
This implies
$$
\mathcal{O}_{\PP(F^{\vee})}(1)=\check{\rho}^*\mathcal{O}_{{\PP}({\EE}^{\vee})}(1) \quad \text{\rm and similarly} \quad 
\mathcal{O}_{\PP(\check{F}^{\vee})}(1)={\rho}^*\mathcal{O}_{{\PP}({\EE})}(1)\, . \eqno(13)
$$
With 
$$
f=c_1(F), \quad \check{f}=c_1(\check{F}), \quad
h=c_1(\mathcal{O}_{{\PP}({\EE})}(1)), \quad\check{h}=c_1(\mathcal{O}_{{\PP}({\EE}^{\vee})}(1))\, ,
$$
this gives the identities of pullbacks of the first Chern classes $c_1({\EE})=-c_1({\EE}^{\vee})=\lambda$
$$
\rho^*(f)+\rho^*(h)=\rho^*u^* (\lambda) = \check{\rho}^* \check{u}^* (\lambda) = 
-\check{\rho}^*(\check{f})-\check{\rho}^*(\check{h})\, .
$$
Since $I={\PP}(F^{\vee})$ over ${\PP}({\EE})$ and 
$\check{\rho}^*\check{h} = c_1({\mathcal O}_{{\PP}(F^{\vee})}(1))$, 
the Chern classes of $F^{\vee} $ and the first Chern class of the 
tautological line bundle satisfy the relation
$$
\check{\rho}^* \check{h}^2 +  \rho^*(f) \, \check{\rho}^*(\check{h}) +\rho^*(c_2(F)) =0 \, .
\eqno(14)
$$
\begin{corollary}\label{up-down} Under the map $\rho_{*}\check{\rho}^*$ we have
$$
\check{h}^2\mapsto h-u^*(\lambda), \quad \check{h} \, \check{u}^*(\xi)\mapsto u^*(\xi), \quad
\check{u}^*(\eta) \mapsto 0
$$
for $\xi \in {\rm CH}^1(\overline{\mathcal{M}}_3)$ and $\eta \in {\rm CH}^2(\overline{\mathcal{M}}_3)$.
\end{corollary}
\begin{proof}
Using relation (14) gives 
$$
\rho_*(\check{\rho}^*(\check{h})^2)=-\rho_{*}(\rho^*(f)\check{\rho}^*(\check{h})-\rho^*(c_2(F))=
-f=h- u^*(\lambda) \, .
$$
The other properties follow from general intersection theory.
\end{proof}

Let now $\psi$ be the class of the codimension $2$ cocycle 
$\check{S} \cdot \check{T}$.
\begin{lemma}\label{up-down-psi}
We have $\rho_{*}\check{\rho}^* \psi= 24 \, h +216 \, \lambda -24\, \delta_0 -48\, \delta_1$.
\end{lemma}
\begin{proof}
By the results of the preceding section we have
$$
\psi= 24 \, \check{h}^2 + 240 \, \check{h}\lambda -24 \, \check{h}\delta_0 -48\, \check{h}\delta_1 +r
$$
with $r \in \check{\rho}^* {\rm CH}^2(\overline{\mathcal{M}}_3)$. Corollary \ref{up-down} implies the result.
\end{proof}

We now claim that the codimension $2$ cycle $\check{S}\cdot \check{T}$ when restricted to the hyperelliptic locus
is of the form $12 \, \check{h}$, in other words, by (2) 
it contains an effective codimension $2$ cycle with class
$$
12\, (9 \, \lambda -\delta_0 - 3\, \delta_1) \, \check{h} + \check{u}^*(\xi)
$$
with $\xi$ a codimension $2$ class on $\overline{\mathcal{M}}_3$.
We check this using the explicit form of the two concomitants $\sigma$ and $\tau$ defining $\check{S}$ and $\check{T}$.
Here $\sigma$ is a polynomial of degree $4$ in $a_0,\ldots,a_{14}$ and degree $4$ in the coordinates $u_0,u_1,u_2$
where $a_0,\ldots,a_{14}$ are the coefficients of the general ternary quartic. 
A calculation shows that $\sigma$ restricted to the locus of
double conics becomes a square $q^2$ with $q$ of degree $2$ in the $u_i$,
while $\tau$ becomes a cube $q^3$. Hence the cycle $S^{\vee} \cdot T^{\vee}$ restricted to the hyperelliptic locus
is represented by an effective cycle representing $6\, q \sim 12 \check{h}$. By Corollary \ref{up-down} under $\rho_* \hat{\rho}^*$ this
is sent to an effective cycle with class $12 (9\, \lambda -\delta_0 -3\,  \delta_1)$. Since $H$ is defined as
the closure of the hypertangent divisor in the generic fibre, the class of $H$ equals $\rho_*\check{\rho}^*\psi$ minus
$12$ times the class of the hyperelliptic locus; by Lemma \ref{up-down-psi} 
we get
$$
24 \, h +216 \, \lambda -24\, \delta_0 -48\, \delta_1 - 12(9\lambda -\delta_0 -3 \delta_1)=
24 \, h +108\, \lambda -12\, \delta_0 - 12 \, \delta_1\, .
$$
We summarize.
\begin{proposition}\label{classhypertangent}
The class $[H]$ of the hypertangent divisor $H$ in ${\PP}_{\overline{\mathcal{M}}_3}({\EE})$
equals  $[\mathcal{O}_{\PP({\EE})}(24)]+ 108 \, \lambda - 12\, \delta_0 -12  \, \delta_1$.
It gives rise to a Siegel modular form
of degree $3$ and weight $(24,0,108)$ vanishing with multiplicity $12$
along the boundary.
\end{proposition}

\smallskip

We now work on the Hurwitz space and define and calculate the class of a hypertangent $H_h$ divisor there. 
It is defined by taking the eight tangent lines at the ramification points of the canonical image. 
More precisely, on $\tilde{\PP}$ we have the line bundle $N$ defined
in (9). Recall that $\tilde{S}_k$ for $1 \leq k \leq 8$ is the pullback of the section $S_k$ 
of $\pi_{9}: \overline{\mathcal{M}}_{0,9}
\to \overline{\mathcal{M}}_{0,8}$. 
Under restriction to the hyperelliptic locus the Weierstrass points 
degenerate to the ramification points. We define the corresponding hypertangent 
divisor $H_h$ in ${\PP}({\EE})$ over $\overline{\mathcal{H}}_{3,2}$ 
by taking the tangents to the canonical image 
of the generic curve at the points of the sections $\tilde{S_k}$, $k=1,\ldots, 8$
over $\mathcal{H}_{3,2}$ and then taking the closure over $\overline{\mathcal{H}}_{3,2}$.  

We now consider the bundle $N(-2\tilde{S}_k)$ on
$\tilde{\PP}$. This line bundle is trivial on the generic fibre of $\pi: \tilde{\PP} \to B$, so
$\pi_*(N(-2\tilde{S}_k))$ is a line bundle on $B$. 
\begin{lemma}
We have 
$$
c_1(R^1\pi_*N(-2\tilde{S}_k))=\Delta_2(k^{+})+\Delta_3(k^{+}), \quad  \text{\rm and} \quad
c_1(\pi_*N(-2\tilde{S}_k))=-\Delta_3(k^{+})-\Delta_3+E_k\, .
$$
\end{lemma}
\begin{proof} Recall that $N=\mathcal{O}(D_k+\pi^*(E_k)-R)$. The first statement follows by analyzing the restrictions over the boundary components. For the second we apply Grothendieck-Riemann-Roch
as in the proof of Proposition \ref{codim2}. By (7) and (8) we have 
$$
\begin{aligned}
c_1(\pi_*N(-2\tilde{S}_k))&= -\Delta_2(k^{+})- 2\, \Delta_3(k^{+})-\Delta_3+c_1(R^1\pi_*N(-2\tilde{S}_k))\\
&= -\Delta_3(k^{+})-\Delta_3+ E_k \, .\\
\end{aligned}
$$
\end{proof} 

Put $\mathcal{F}_k=\pi_*(N(-2\tilde{S}_k))$.
The injection $N(-2\tilde{S}_k) \hookrightarrow N$ induces an injection $\mathcal{F}_k \to {\EE}$.
Pulling back to $\PP({\EE})$ via $u^*$ and composing with the canonical surjection
$u^*({\EE}) \to \mathcal{O}_{\PP({\EE})}(1)$ we get an induced map
$$
q: u^*\mathcal{F}_k \to \mathcal{O}_{\PP({\EE})}(1)\, .
$$
The degeneracy locus of $q$ is an effective divisor $F_k$ that is the vanishing divisor
of a section of $\mathcal{O}_{\PP({\EE})}(1) \otimes u^*\mathcal{F}_k^{-1}$. The interpretation is as follows.
The map $\phi$ defines an embedding of the generic fibre of $\tilde{\PP}$ into the generic fibre ${\PP}({\EE})$.
If we identify $H^0({\PP}^1,\mathcal{O}(2))$ with the fibre of ${\EE}$ and projectivize, 
the divisor $p_1+p_2 \in |\mathcal{O}(2)|$ is mapped to
the line through through the points $\phi(p_1)$, $\phi(p_2)$.
We now sum these divisors $F_k$ and get an effective divisor $H_h$ with class
$$
\begin{aligned}
{}[H_h]=8\, [\mathcal{O}(1)]- \sum_{k=1}^8 [u^*\mathcal{F}_k] &= 
[\mathcal{O}(8)] +u^*(3\, \Delta_3+8 \, \Delta_3 - \sum_{k=1}^8 E_k) \\
 &= [\mathcal{O}(8)] +u^*(8\, \lambda -2\, \Delta_2 +\Delta_3)\, ,\\
\end{aligned}
$$
where we use the formulas of Section \ref{extending} and Section \ref{symmetrization}.

We can now compare the class of the hyperelliptic hypertangent divisor $H_h$ with that of 
the pull back of the hypertangent divisor $H$ to the Hurwitz space.  By Proposition \ref{classhypertangent}
the pull back of $H$ has class $[\mathcal{O}(24)] + 108\, \lambda - 24\, (\Delta_2+\Delta_4)- 12 \, \Delta_1$.
Since the $24$ Weierstrass points collapse with multiplicity $3$ to the $8$ ramification points 
we compare the class (of the pull back of)  $[H]$  with that of $3\, [H_h]$. 
Substituting the formula for $\lambda$ we get 
$$
[H]-3\, [H_h]=9\, \Delta_1\, ,
$$
which means that the pull back of $H$ vanishes with multiplicity $9$ 
at the hyperelliptic boundary component $\Delta_1$. 
\end{section}
\begin{section}{Genus Four}
For a smooth curve $C$ of genus $4$ the natural map ${\rm Sym}^2(H^0(C,\omega_C)) 
\to H^0(C,\omega_C^{\otimes 2})$ is surjective and the kernel has dimension $1$.
It determines a quadric in ${\PP}^{3}={\PP}(H^0(C,\omega_C))$ containing the canonical curve. Over
$\overline{\mathcal{M}}_4$ we find a corresponding exact sequence
$$
0\to U \to {\rm Sym}^2({\EE}) \to \pi_* \omega_C^{\otimes 2} \to 0 \, .
$$
The line bundle $U$ has first Chern class $5\, \lambda - (13\, \lambda -\delta)=-8\, \lambda +\delta$
by Mumford's calculation of $c_1(\pi_*\omega_c^{\otimes 2})$ \cite[Thm.\ 5.10]{Mumford1977}. In the bundle ${\PP}({\EE})$ the quadric containing the canonical curve
determines a divisor $Q$. 
Let $u: {\PP}({\EE}) \to \overline{\mathcal{M}}_4$ be the projection.

\begin{lemma}
The divisor class of $Q$ satisfies: $[Q]=[\mathcal{O}(2)]+ u^*(8\, \lambda -\delta)$. 
\end{lemma}
\begin{proof} 
Observe that $u^* u_*\mathcal{O}(2)=u^*({\rm Sym}^2({\EE}))$.
The natural morphism $u^* u_*\mathcal{O}(2) \to \mathcal{O}(2)$ induces $u^*U \to \mathcal{O}(2)$.
The divisor $Q$ is the vanishing locus of this morphism, hence has class
$[\mathcal{O}(2)]+u^*(8\, \lambda -\delta)$.
\end{proof}
\begin{corollary}
The effective divisor $Q$ defines a Teichm\"uller modular cusp form $\chi$ 
of genus $4$ and
weight $(2,0,0,8)$.
\end{corollary}
If we view a section of ${\rm Sym}^2({\EE})$ as a quadratic form 
on ${\EE}^{\vee}$ we can take the discriminant, cf.\ \cite{CG}. 
Doing this with the form $\chi$ of
weight $(2,0,0,8)$ just constructed we get a scalar-valued modular form
$D(\chi)$ of weight $34$. This modular form vanishes on the closure of the locus
of curves whose canonical model lies on a quadric cone. This locus has
class $34 \lambda - 4 \, \delta_0 - 14\, \delta_1 - 18 \, \delta_2$ by
Teixidor i Bigas \cite[Prop. 3.1]{Teixidor}
and equals the divisor of curves with a vanishing thetanull. 
The modular form $D(\chi)$ 
is the square root of the restriction to $\mathcal{M}_4$ of the product of the even theta characteristics on $\mathcal{A}_4$.
\end{section}
\begin{section}{Appendix on base-point freeness}\label{app}
The relative dualizing sheaf $\omega_{\pi}$ of the universal family   $\pi :{\mathcal C}_g  \to {\mathcal M}_g$ of genus $g$ smooth curves  is base point free and the surjection
$\pi^*{\EE} \to \omega_{\pi}$ gives a map 
$\varphi: {\mathcal C}_g \to {\PP}(\EE)$ over 
${\mathcal M}_g$, which is generically an embedding.
Let $\Gamma$ be the image $\varphi({\mathcal{C}}_g)$. We wish to describe the closure of the image over the generic points of the boundary
components $\Delta_i$ for $i=0,\ldots,[g/2]$.
Over the general point of $\Delta_0$ the sheaf 
$\omega_{\pi}$ is base point free and the map $\varphi$ extends over this locus.
But over the general point of $\Delta_i$, $i\geq 1$, which represents a nodal curve of the form 
 $C_1\cup C_{2}$, with $C_1$, $C_{2}$ smooth curves of 
genus $i$ and $g-i$ meeting at a nodal point $x$,  
the sheaf  $\omega_{\pi}$ has a base point at $x$. 
We consider a family $\pi: Y \to B$ of stable curves of genus $g$ with $B$ the 
spectrum of a discrete valuation ring. 
We assume that the central fibre $C$ is a nodal curve
$C=C_1\cup C_{2}$ of genera $i$ and $g-i$ 
and smooth generic fibre. After a degree $2$
base change $B'\to B$ we get an $A_1$-singularity which we resolve
resulting in a semistable family $\pi': X \to B'$ with a special fibre
which is a chain of three curves 
$C=C_1^{\prime} \cup R \cup C_{2}^{\prime}$ with 
$C_1^{\prime}\cong C_1$ and $C_2^{\prime}\cong C_2$  
smooth curves of genus $i$ and $g-i$ and $R$ a rational $(-2)$-curve. 
We have the commutative diagram
$$
\xymatrix{
X \ar[r]^{v}  \ar[d]_{\pi' }   & Y  \ar[d]^{\pi} \\
 B' \ar[r]_{\sigma} & B 
}
$$
The morphism $v$ is $(2:1)$ ramified at $C_1, C_{2}$. 
We have  $v^*\omega_{\pi} = \omega_{\pi'}$,
and  $\sigma^*\EE_B = \EE_{B'}$ and $v^*C_j = 2\, C'_j+R$ for $j=1,2$.
There is then a natural $(2:1)$ map $\PP(\EE_{B'}) \to {\PP}(\EE_B)$.

\smallskip

Now we will show that the system $\omega_{\pi'}(-R)$ 
defines a map  $X\to \PP(\EE_{B'})$ which combined with the 
above  $(2:1)$ map  gives a $(2:1)$ map     
$\varphi' : X \to \PP(\EE_B)$
mapping  the curves  $C'_1$ and $C_2'$ 
to their canonical image and  $R$ to a double line. 
The reduced image of the map $\varphi'$ 
describes the closure of $D$ over $b_0$, the special point of $B$.

To avoid unnecessary notation we now write $\pi: X \to B$ for the semi-stable family denoted
by $\pi': X \to B'$ above. 

\begin{proposition}\label{bpfreeness}
Let $\omega$ be the relative dualizing sheaf of ${\pi}: X \to B$. Then we have
 ${\pi}_*(\omega(-R))\cong \pi_*(\omega)$ and the central fibre of $\pi_*(\omega(-R))$ is of codimension $1$ in $H^0(C,\omega(-R))$
and defines a base point free linear system on $C$.
\end{proposition}
\begin{proof}
We let $q=C_1\cap R$ and $p=C_2\cap R$.
The exact sequence $0\to \omega(-R) \to \omega \to \omega_{|R}\to 0$ induces a sequence
$$
0 \to \pi_*(\omega(-R)) \to \pi_*(\omega) \xrightarrow{r} \pi_*(\omega)_{|R}
$$
and the map $r$ is zero because $\omega_{|C}=(\omega_{C_1}(q),\mathcal{O}_R,\omega_{C_2}(p))$,
therefore the restrictions to $C_1$ (resp.\ $C_2)$ must vanish at $q$ (resp.\ $p$), hence extend
by $0$ on $R$. We thus see by the exactness that $\pi_*(\omega(-R))\cong \pi_*(\omega)$. 

Next we observe that $\dim H^0(C,\omega(-R))=g_1+g_2+1$ with $g_i$ the genus of $C_i$. This 
follows directly from $\omega(-R)_{|C}=(\omega_{C_1}, \mathcal{O}_R(2), \omega_{C_2})$.

We have the exact sequence
$$
0 \to \omega(-R-C_1) \to \omega(-R) \to \omega(-R)_{|C_1}\to 0\, , \eqno(15)
$$
where $\omega(-R)_{|C_1}\cong \omega_{C_1}$ and 
$\omega(-R-C_1)_{|C}=(\omega_{C_1}(q),\mathcal{O}_R(1), \omega_{C_2})$. 
For a section $(s_1,s,s_2) \in H^0(C, \omega(-R-C_1))$ the section $s$ is the unique
section of $\mathcal{O}_R(1)$ that vanishes at $q$ and with $s(p)=s_2(p)$. We thus see
$\dim H^0(C, \omega(-R-C_1))=g_1+g_2$. Therefore $h^0(\omega(-R-C_1))$ has constant 
rank $g_1+g_2$ on the fibres of $\pi$, hence $R^1\pi_*(\omega(-R-C_1))$ is a line bundle.
We conclude that the special fibre of $\pi_*(\omega(-R-C_1))$ equals $H^0(C,\omega(-R-C_1))$.
But $\pi_*(\omega(-R)_{|C_1})$ is a torsion sheaf, hence the connecting homomorphism
$\pi_*(\omega(-R)_{|C_1})\to R^1\pi_*(\omega(-R-C_1))$ of (15)  must be zero and 
we get an induced exact sequence
$$
0 \to \pi_*(\omega(-R-C_1)) \xrightarrow{\iota} \pi_*(\omega(-R)) 
\xrightarrow{j} \pi_*(\omega(-R))_{|C_1}) \to 0 \, . 
$$

Consider now a section $\sigma$ of $\pi_*(\omega(-R-C_1))$ with restriction $(s_1,s,s_2)$
to $C$. Suppose that $s\neq 0$.
If we multiply $\sigma$ with a local section $\tau$ of $\mathcal{O}(C_1)$ on $X$
with divisor $C_1$ then $\iota(\sigma)=\sigma\cdot \tau_{|C}$ has as restriction to $R$
a section of $\mathcal{O}_R(2)$ vanishing with multiplicity $2$ at $q$ and therefore
it does not vanish anywhere else. Hence the subspace of the special fibre $V$ of 
$\pi_*(\omega(-R))$ of sections vanishing on $C_1$ has $q$ as only base point on $R$.
Furthermore, the map $j$ is surjective, and choosing a section $s_1 \in H^0(C_1,\omega_{C_1})$
with $s_1(q)\neq 0$ we see that $q$ is not a base point. 
Therefore there are no base points on $R$. 
By the surjectivity
of $j$ the restriction of $V$ to $C_1$ is $H^0(C_1,\omega_{C_1})$ and therefore
there are no base points on $C_1$. By symmetry the same holds for $C_2$.

Similarly to (15) we have an exact sequence 
$$
0 \to \omega(-R-C_1-C_2) \to \omega(-R) \to \omega(-R)_{|C_1+C_2}\to 0
$$
and by a similar reasoning we see that we get an exact sequence
$$
0 \to \pi_*(\omega(-R-C_1-C_2)) \xrightarrow{\iota} \pi_*(\omega(-R)) 
\xrightarrow{j} \pi_*(\omega(-R))_{|C_1+C_2}) \to 0 \, . 
$$
This implies that given $s_1\in H^0(C_1,\omega_{C_1})$ and $s_2 \in H^0(C_2,\omega_{C_2})$
there is a unique element $(s_1,s,s_2)$ in the special fibre $V$ of $\pi_*(\omega(-R))$
mapping to $(s_1,s_2)$ under $j$. The morphism $X \to {\PP}(\pi_*(\omega(-R)))$ is given by
the surjection $\pi^*\pi_*(\omega(-R)) \to \omega(-R)$. The image of the curve $C$ in the special
fibre of ${\PP}({\EE})$ consists of the canonical images of $C_1$ and $C_2$, provided with
with image of $p$ and $q$ and the image of $R$, that is, the line connecting the images of $p$ and $q$.
If the genus $g(C_i)=1$ then the image of $C_i$ is a point.
\end{proof}
\end{section}
\begin{section}{Appendix: divisor classes of Gheorghita-Tarasca and Korotkin-Sauvaget-Zograf}\label{G-T-K-S-Z}
Here we apply the method employed in Section \ref{hypertangentsection} to determine in
a relatively straightforward way the divisor classes of two divisors in ${\PP}({\EE}^{\vee}_k)$ with
${\EE}_k=\pi_*(\omega_{\pi}^k)$,
thus reproving a theorem of Gheorghita-Tarasca (\cite[Thm 1]{G-T}) 
and a theorem of Korotkin-Sauvaget-Zograf (\cite[Thm.\ 1.12]{KSZ}).
The first divisor is a generalization of a divisor in ${\PP}({\EE}^{\vee})$ 
considered by Gheorghita in \cite{Gheorghita}. We consider in ${\PP}({\EE}^{\vee}_k)$ over $\mathcal{M}_g$
the divisor
$$
G_k=\{ (C,\omega) \in {\PP}({\EE}_k^{\vee}): \text{${\rm div}(\omega)$ contains a Weierstrass point}\}
$$
and let $\overline{G}_k$ be the closure of $G_k$ in ${\PP}({\EE}^{\vee}_k)$ over 
$\overline{\mathcal{M}}_g$. We let $\check{u}: {\PP}({\EE}_k^{\vee})\to \overline{\mathcal{M}}_g$
be the natural morphism and $\check{h}$ the hyperplane class on ${\PP}({\EE}_k^{\vee})$.

\begin{theorem}\label{G-T-Thm} {\rm (Gheorghita-Tarasca)}
The class of $\overline{G}_k$ is given by
$$
\frac{1}{k}[\overline{G}_k]= g(g^2-1) \check{h} + 2(3g^2+2g+1)\check{u}^*\lambda -   \binom{g+1}{2} 
\check{u}^* \delta_0 - \sum_{i=1}^{[g/2]} (g-i)i(g+3) \check{u}^*\delta_i\, . 
$$
\end{theorem} 
The second divisor is the divisor $Z_k$ in ${\PP}({\EE}_k^{\vee})$ over $\overline{\mathcal{M}}_g$
of regular $k$-differentials for $k\geq 2$ possessing a double
zero.
\begin{theorem}\label{K-S-Z-Thm} {\rm (Korotkin-Sauvaget-Zograf)}
The class of the divisor $Z_k$ for $k\geq 2$ for $k\geq2$ and $(g,k)\neq (2,2)$
is given by
$$
[Z_k]= (4k+2)(g-1)\, \check{h} + k(k+1) \, \check{u}^*(12\lambda -\sum_{i=0}^{[g/2]} \delta_i)\, .
$$
\end{theorem}
For the proof of both theorems we use,  as in Section \ref{hypertangentsection},  
the incidence variety $I_k$ between $\PP(\EE_k)$ and   $\PP(\EE^{\vee} _k)$ 
which fits in the following commutative diagram:
$$
\xymatrix{
I_k  \ar[d]_{\rho}  \ar[r]^{\check{\rho}}  & \ar[d]^{\check{u}} {\PP}(\EE_k^{\vee}) \\
  \PP(\EE_k)   \ar[r]^u  &   \overline{\mathcal M}_g   }
$$
We denote by $h$ (resp.\ $\check{h}$) the first Chern class of the hyperplane line  
bundle on  ${\PP}(\EE_k) $ (resp.\  ${\PP}(\EE_k^{\vee})$), thus suppressing the dependence on $k$.
 As explained in Section \ref{hypertangentsection},   
we have  $I_k={\PP}(\check{F}_k^{\vee})$ as a bundle over  $ {\PP}(\EE_k^{\vee})$, 
with $\check{F_k}$ defined by the exact sequence on ${\PP}({\EE}_k^{\vee})$
$$
0\to \check{F_k}  \to \check{u}^* \EE_k^{\vee}  \to \mathcal{O}_{\PP(\EE_k^{\vee})}(1) 
\to 0.
$$
Then  $\rho^*h= {\mathcal O}_{\PP(\check{F_k}^{\vee})}(1)$. 
Similarly,  $I_k=\PP(F^{\vee}_k)$ as a bundle over  $\PP(\EE_k)$, 
with $F_k$ the tautological rank $r-1$ bundle on  $ \PP(\EE_k)$. Then  
$\check{\rho}^*(\check{h}) ={\mathcal O}_{\PP({F_k}^{\vee})}(1)$.

We let 
$$
\gamma=\check{\rho}_*\rho^*: {\rm CH}_{\QQ}^*({\PP}({\EE}_k))\to 
{\rm CH}_{\QQ}^*({\PP}({\EE}_k^{\vee}))
$$ be the induced map.

\begin{lemma} \label{gamma}
We have 
$\gamma(h^i)=0$ for  $i\leq r-3$, $ \gamma(h^{r-2})=1$ and  
$\gamma( h^{r-1})  = \check{h}+\check{u}^*{\rm c}_1(\EE_k)$.
\end{lemma}
\begin{proof} For dimension reasons
 $\gamma(h^i)=0$ for  $i\leq r-3$. Moreover $ \gamma(h^{r-2})=1$ by construction. 
Applying $\check{\rho}_*$
to the Chern class relation 
$ \sum_{i=0}^{r-1}    (-1)^{r-1-i}  (\rho^*h)^i  \check{\rho}^* c_{r-1-i}(\check{F}_k^{\vee})=0$
we get  $ \gamma( h^{r-1})  =   c_1 (\check{F}_k^{\vee})$ and this equals $\check{h}+\check{u}^*{\rm c}_1(\EE_k)$ by the exact sequence.
\end{proof}

\noindent 
{\sl Proof of Theorem \ref{G-T-Thm}.} 
Let $W$ be the Weierstrass divisor on $\overline{\mathcal{C}}_g$. This is an
irreducible divisor. We denote by $\varphi_k: \overline{\mathcal C}_g \to {\PP}({\EE}_k)$ the 
morphism defined in Section \ref{classkcanonical}. 
For $k\geq 2$ we have $G_k=\rho_*\rho^*(\varphi_k(W))$ over $\mathcal{M}_g$
and $\rho_*\rho^*$ sends an irreducible divisor over $\overline{\mathcal{M}}_g$ 
to an irreducible divisor. Therefore we have
$[\overline{G}_k]=\gamma({\varphi_k}_*[W])$.
The group 
${\rm Pic}_{\QQ}(\overline{\mathcal{C}}_g)$ is generated by
$\omega_{\pi},\pi^*\lambda$, $\pi^*\delta_0$ and $\gamma_i$ (for $i=1,\ldots,g-1$) 
with $\gamma_i$ the divisor class defined by the component 
of genus $i$ lying over $\Delta_{\min(i,g-i)}$.
By Cuckierman \cite{Cuckierman} the divisor class
$[W]$ can be written as $w_1-w_2$ with 
$$
w_1=\binom{g+1}{2}\omega_{\pi} -\pi^*\lambda, \quad
w_2=\binom{g-i+1}{2} \sum_{i=1}^{g-1} \gamma_i\, .
$$
Note that $\varphi_k^*h=k\, \omega_{\pi}$. We thus get
$$
\begin{aligned}
{\varphi_k}_*w_1&= {\varphi_k}_*(\frac{g(g+1)}{2} \omega_{\pi}-\pi^*\lambda)=
{\varphi_k}_*\varphi_k^*\left( \frac{g(g+1)}{2k} h -u^*\lambda\right) \\
&=\left( \frac{g(g+1)}{2k} h-u^*\lambda\right)\, {\varphi_k}_*[1]= 
\left(\frac{g(g+1)}{2k} h-u^*\lambda\right)(\sum_{i=0}^{r-2} h^i \beta_{r-2-i}) \\
&= \sum_{i=0}^{r-1} h^iu^*(\frac{g(g+1)}{2k} \beta_{r-1-i} -\lambda\beta_{r-2-i})\, , \\
\end{aligned}
$$
where we used (1).
Lemma \ref{gamma} implies that under applying $\gamma$ 
only the terms with $h^j$ where $j=r-1$ or $r-2$ contribute 
and we get
$$
\gamma({\varphi_k}_*w_1)=\gamma(h^{r-1})\frac{g(g+1)}{2k}\beta_0 +
\gamma(h^{r-2})\check{u}^*(\frac{g(g+1)}{2k} \beta_1-\lambda\beta_0) \, .
$$
Substituting the expressions for $\beta_0$ and $\beta_1$ from Proposition \ref{curveclass}
leads to
$$
\gamma({\varphi_k}_*w_1)= g(g^2-1)\, \check{h}+ 2\, k\, (3g^2+2g+1)\check{u}^*\lambda -  
{k\over 2}\, g(g+1) \sum_{i=0}^{[g/2]} \check{u}^* \delta_i \, .
$$
For the term $\gamma({\varphi_k}_*w_2)$ we first observe 
$\gamma({\varphi_k}_*\gamma_i)=(2i-1)k\check{u}^*\delta_i$ 
because the component of genus $i$ over $\Delta_{\min(i,g-i)}$ 
has degree $(2i-1)k$ in ${\PP}^{r-1}$ and thus
maps under $\gamma$ 
to $(2i-1)k$ times the class $[1]$ of $({\PP}^{r-1})^{\vee}$ over $\Delta_{\min(i,g-i)}$. 
This gives
$$
\begin{aligned}
\gamma( \varphi_{k*} w_2) &= \frac{k}{2}  \sum _{i=1}^{g-1} (g-i)(g-i+1)(2i -1)\, 
\check{u}^*\delta_i \\
&= \frac{k}{2}  \sum _{i=1}^{[g/2] }( (g-i)(g-i+1) (2i-1) +i(i+1)(2g-2i-1))  \check{u}^*\delta_i \\
& = \frac{k}{2}  \sum _{i=1}^{[g/2] }( 2i(g-i)(g+3)-g(g+1))   \check{u}^*\delta_i . 
\end{aligned}
$$
Together this gives the correct expression for class of $[\overline{G}_k]$
as in Theorem \ref{G-T-Thm}.

When $k=1$, over ${\mathcal M}_g \cup \Delta_0$ we work as above and 
the coefficients of $\lambda $ and $\delta_0$ in the formula are as in the case  $k\geq 2$.
To find the contribution of $\delta_i$ in the formula of $\beta_1$, we work  
over the family over a base  $B$, as in Appendix \ref{app},  
 where  we have the  $(2:1)$ morphism $\varphi' : X \to \PP(\EE)$  
defined by the  $\omega_{\pi'}(-R)$.
 We follow the notation of Appendix \ref{app} and in the formulas  
we  only need to consider terms that contribute to the boundary class $\delta_i$.  
By \cite{Cuckierman} the Weierstrass divisor 
does not pass through the node of a general element 
over $\Delta_i$ and thus  $v^*W$ does not contain the 
`exceptional' divisor $R$.
We have by Cuckierman's formula $[v^*W]=w_1-w_2$ with $w_1=\binom{g+1}{2}\omega_{\pi'}-{\pi'}^*\lambda$,
and as contribution to $w_2$ over $\Delta_i$
(for $i\leq [g/2]$) we have the expression
$$
i(i+1)(2\gamma_1+\mathfrak{r})+(g-i)(g-i+1)(2\gamma_2+\mathfrak{r})\, , \eqno(16)
$$
where $\gamma_1$ (resp.\ $\gamma_2$) is the class of the component $C_1'$ of genus $i$ 
(resp.\ $C_2'$ of genus  $g-i$) over $\Delta_i\cap B$ and $\mathfrak{r}$ the class of $R$. 

If we denote by $T$ the closure of the reduced  image of $v^*W$ under the $(2:1)$ map 
$\varphi' : X' \to \PP(\EE_B)$, then $[G_1]= \gamma([T])$.
Recall that $\varphi'^*h = \omega_{\pi'}-R$. Thus the $\delta_i$-contribution in $2[T]$
coming from from $w_1$ is
$$
\frac{g(g+1)}{2} \varphi^{\prime}_*(\omega_{\pi'})=
\frac{g(g+1)}{2} \varphi^{\prime}_*({\varphi^{\prime}}^*h+\mathfrak{r})=
\frac{g(g+1)}{2}\left(h\varphi^{\prime}_*[1]+\varphi^{\prime}_*\mathfrak{r})\right)\, .
$$
If we apply (1) the contribution to $\delta_i$ in
$$
\gamma(h{\varphi'}_*[1])=\gamma(\sum_{i=0}^{g-2} h^{i+1} u^*\beta_{g-2-i}) 
= 2(g-1) (\check{h}+\check{u}^*\lambda) +\check{u}^*\beta_1 
$$ 
comes from $\check{u}^*\beta_1$ alone and equals $-4 \check{u} ^*\delta_i$,
as $\delta_i$ appears in the formula of $\beta_1$ with coefficient $-2$. 
Since $\gamma(\varphi^{\prime}_*\mathfrak{r})=2\, \check{u}^*\delta_i$
we get from $w_1$  together  the contribution $-g(g+1)\check{u}^*\delta_i$.
From $w_2$ we get by applying $\gamma$ to (16), using $\gamma(\varphi^{\prime}_*\gamma_1)=
(2i-2)\check{u}^*\delta_i$, $\gamma(\varphi^{\prime}_*\gamma_2)=
(2g-2i-2)\check{u}^*\delta_i$ and $\gamma(\varphi^{\prime}\mathfrak{r})=2\check{u}^*\delta_i$,
the contribution $2(g+3)i(i-g)+g(g+1)$. Together  $w_1-w_2$ thus contribute $-2(g+3)i(i-g)$
to the coefficient of $\delta_i$, as required.

\hfill$\square$

\smallskip

\noindent
{\sl Proof of Theorem \ref{K-S-Z-Thm}.} Here $k\geq 2$, hence we have the morphism
$\varphi_k: \overline{\mathcal C}_g  \to \PP(\EE_k)$.
Let $\pi_1: \overline{\mathcal C}_{g,1} \to  \overline{\mathcal C}_g $ 
 be the universal curve over $\overline{\mathcal C}_g$ and  
$s: \overline{\mathcal C}_g \to   \overline{\mathcal C}_{g,1}$ the tautological section 
the image of which we denote by $S$.

We claim: $\varphi_k^*F_k = \pi_{1*}(\omega^{\otimes k}_{\pi_1}( -S))$. 
Indeed, by our assumptions on $g$ and $k$, we have $R^1\pi_{1*}(\omega^{\otimes k}_{\pi_1}( -S))=0$, so
$\pi_{1*}(\omega^{\otimes k}_{\pi_1}( -S))$ is a vector bundle on $\overline{\mathcal C}_g$.
For a  point $x\in \PP(\EE_k)$ the fibre of $F_k$ is the hyperplane in the corresponding fibre of $\EE_k$ 
representing the point $x$.
When $x=\varphi_k (p)$ with $ p\in \overline{\mathcal C}_g$, then $(F_k)_x=H^0(C_p,\omega^{\otimes k}_{\pi_1}(-p))$, with $C_p$ the corresponding firbre of $\pi_1$ over $p$.  Hence the claim.
We now have on $\overline{\mathcal C}_g $ the sequence
$$
0 \to \pi_{1*}(\omega^{\otimes k}_{\pi_1}( -2S)) \to \pi_{1*}(\omega^{\otimes k}_{\pi_1}( -S)) 
\to s^* (\omega^{\otimes k}_{\pi_1} ( -S))\to 0 \, , \eqno(17)
$$
with $s^*(\omega_{\pi_1}^{\otimes k}(-S))\cong \omega_{\pi}^{k+1}$, and this sequence is
exact up to codimension $2$ because  $R^1 \pi_{1*}(\omega^{\otimes k}_{\pi_1}( -2S))$ 
vanishes in codimension $2$ in view of
the conditions on $(g,k)$.

Let $F_k(\nu)=\pi_{1*}(\omega^{\otimes k}(-\nu S))$ for $\nu=1,2$ with $F_k(1)\cong\varphi^*F_k$.  
Let  $j: \PP(F_k(1)^{\vee} ) \to \PP(F_k^{\vee})$ be the natural map.     
Then $\tilde{h}=j^*\check{\rho}^*(\check{h})$ is the class of the hyperplane line  
bundle on $\PP(F_k(1)^{\vee})$.
  The  inclusion  $F_k(2) \hookrightarrow  F_k(1) $ induces a map 
$\sigma: \PP(F_k(2)^{\vee}) \to \PP(F_k(1)^{\vee})$.
 We have the commutative diagram:
$$
\xymatrix{
\PP(F_k(2)^{\vee}) \ar[dr] \ar[r]^{\sigma} & \PP(F_k(1)^{\vee})\ar[d]_{\rho_1}  \ar[r]^{j} & I 
=\PP(F_k^{\vee}) \ar[d]_{\rho}  \ar[r]^{\check{\rho}}  & \ar[d]^{\check{u}}   \PP(\EE_k^{\vee}) \\
& \overline{\mathcal C}_g  \ar[r]^{\varphi_k} & \PP(\EE_k)   \ar[r]^u  &   \overline{\mathcal{M}}_g    }
$$
Let $\alpha= j \sigma :  \PP(F_k(2)^{\vee}) \to I$ and      
$A= {\rm Im}(\alpha)$ the image of this map. Then $[Z_k] =\check{\rho}_*[A]$.
By the  exact sequence (17)  we have
$ \sigma _*[\PP(F_k(2)^{\vee})] =\tilde{h} + (k+1)  \rho_1^* \omega_{\pi} $.
We observe that $j_*[1]=\rho^*[\overline{\Gamma}_k]$ and  $k\, \omega_{\pi}= \varphi_k^* h$ and we find
$$
\begin{aligned}
k{[A]} &= k\, j_*(j^*\check{\rho}^*\check{h})  + (k+1) j_*(\rho_1^*\phi_k^* h)  = 
k\, \check{\rho}^*\check{h} \, j_*[1]  + (k+1) \rho^*(h) j_*[1]\\
&= k\, \check{\rho}^*\check{h}\, \rho^*[\overline{\Gamma}_k] +  (k+1) \rho^*(h[\overline{\Gamma}_k])\, .\\
\end{aligned}
$$
By Proposition \ref{curveclass} and Lemma \ref{gamma} we have
$\check{h} \gamma([\overline{\Gamma}_k])= 2k(g-1) \check{h}$ and
$$
\gamma (h[\overline{\Gamma}_k])= \gamma(\sum_{i=0}^{r-2} h^{i+1} u^*\beta_{r-2-i}) 
= 2k(g-1) ( \check{h}+\check{u}^*{\rm c}_1(\EE_k)) +\check{u}^*\beta_1
=  2k(g-1)  \check{h}+k^2\check{u}^*\kappa_1
$$
and thus
$
[Z_k]= 2(2k+1)(g-1)\check{h}  +k(k+1)\check{u}^* \kappa_1
$
in agreement with the formula of Theorem \ref{K-S-Z-Thm}.
\hfill$\square$
\end{section}

\end{document}